\documentclass{article}
\usepackage{fullpage}
\usepackage[colorlinks,backref]{hyperref}
\usepackage{graphicx}
\usepackage{amsmath,amsthm,amssymb,enumerate}
\usepackage[normalem]{ulem} 
\usepackage{euscript,mathrsfs}
\usepackage[left=3cm,right=3cm,top=3.5cm,bottom=3.5cm]{geometry}
\usepackage{color}
\usepackage{bm}
\usepackage{bbm}
\catcode`\@=11 \@addtoreset{equation}{section}

\catcode`\@=12
\usepackage{color}
\allowdisplaybreaks

\newtheorem{Theorem}{Theorem}[section]
\newtheorem{Proposition}[Theorem]{Proposition}
\newtheorem{Lemma}[Theorem]{Lemma}
\newtheorem{Corollary}[Theorem]{Corollary}

\theoremstyle{definition}
\newtheorem{Definition}[Theorem]{Definition}

\newtheorem{Remark}[Theorem]{Remark}

\newcommand{\bTheorem}[1]{
	\begin{Theorem} \label{T#1} }
	\newcommand{\eT}{\end{Theorem}}

\newcommand{\bProposition}[1]{
	\begin{Proposition} \label{P#1}}
	\newcommand{\eP}{\end{Proposition}}

\newcommand{\bLemma}[1]{
	\begin{Lemma} \label{L#1} }
	\newcommand{\eL}{\end{Lemma}}

\newcommand{\bCorollary}[1]{
	\begin{Corollary} \label{C#1} }
	\newcommand{\eC}{\end{Corollary}}

\newcommand{\bRemark}[1]{
	\begin{Remark} \label{R#1} }
	\newcommand{\eR}{\end{Remark}}

\newcommand{\bDefinition}[1]{
	\begin{Definition} \label{D#1} }
	\newcommand{\eD}{\end{Definition}}

\newcommand{\dif}{\mathrm{d}}
\newcommand{\Div}{{\rm div}}
\newcommand{\Grad}{\nabla_x}

\newcommand{\dx}{\,{\rm d} {x}}

\newcommand{\dt}{\,{\rm d} t }

\newcommand{\dv}{\rm div}
\newcommand{\R}{\mathbb{R}}
\newcommand{\vu}{\vc{u}}
\newcommand{\vc}[1]{{\bf #1}}

\newcommand{\mr}{\mathbb{R}}
\newcommand{\prst}{\mathbb{P}}
\newcommand{\p}{\mathbb{P}}
\newcommand{\T}{\mathbb{T}}

\newcommand{\mn}{\mathbb{N}}

\newcommand{\mt}{\mathbb{T}^3}
\newcommand{\D}{{\rm d}}

\begin{document}


\title{Convergence of a spectral method for the stochastic incompressible Euler equations}

\author{Abhishek Chaudhary 
}

\date{\today}

\maketitle

\medskip
\centerline{ Centre for Applicable Mathematics, Tata Institute of Fundamental Research}
\centerline{P.O. Box 6503, GKVK Post Office, Bangalore 560065, India}\centerline{abhi@tifrbng.res.in}


\medskip
\begin{abstract}
	We propose a spectral viscosity method (SVM) to approximate the incompressible Euler equations driven by a \emph{multiplicative} noise. We show that SVM solution converges to a {\em dissipative measure-valued martingale} solution. These solutions are weak in the probabilistic sense i.e. the probability space and the driving Wiener process are an integral part of the solution. We also exhibit weak (measure-valued)-strong uniqueness principle. Moreover, we establish \emph{strong} convergence of approximate solutions to the regular solution of the limit system at least on the lifespan of the latter, thanks to the weak (measure-valued)--strong uniqueness principle for the underlying system.
\end{abstract}

\section{Introduction}
Fluid dynamics is one of the most demanding research area in mathematics and motivates many question in stochastic analysis. Since the equations of turbulence are very difficult to examine, many reseachers are interested to study the so called \emph{shell-models}. These are simplified models which capture some of the phenomena of turbulance in a more tractable mathematical reference. One typical example is the Euler equations for the motion of an inviscous incompressible fluid which have an intensive role in geophysics;  in science; in meteorology; in engineering; in areospace; in astrophysics and of course, in mathematics where advanced techniques for existence and uniqueness provide important mathematical tool and new theoretical insight. Stochastic partial differential equations (SPDEs) is a subject that has been the focus of much activity during the last decade. One important motivation for studying such equations comes from limiting models of particle system. Stochastic deformation of classical mechanics is a challenging area in which interactions with stochastic analysis are substantial. To accommodate external influence for which a precise model is missing, it is natural to consider stochastic version of the Euler equations.

\noindent 

In this article, we consider the stochastic Euler equations governing the time evolution of the velocity $\vu$ and the scalar pressure field $\Pi$ of an inviscid fluid on the three dimensional torus $\T^3$.
The system of equations reads
\begin{equation}
	\begin{cases} 
		\D \vu(t,x) + \left[ \Div (\vu(t,x) \otimes \vu(t,x))+  \Grad \Pi(t,x)\right]  \dt  =  \mathcal{\sigma} (\vu(t,x)) \,\D W(t), &\quad \mbox{in}\,\, (0,T)\times\T^3, \label{P1} \\ 
		\Div \,\vu(t,x) = 0,&\quad \mbox{in}\,\, (0,T)\times\T^3, \\
		\vu(0, x) = \vu_0(x), &\quad \mbox{in}\,\, \mathbb{T}^3 
	\end{cases}
\end{equation}
where $T>0$ fixed, $\vu_0$ is given initial data. 
 Let $\big(\Omega, \mathcal{F}, (\mathcal{F}_t )_{t\ge0}, \mathbb{P} \big)$ be a stochastic basis, where $\big(\Omega, \mathcal{F}, \mathbb{P} \big)$ is a probability space and $(\mathcal{F}_t)_{t \ge 0}$ is a complete filtration with the usual assumptions. We assume that $W$ is a cylindrical Wiener process defined on the probability space $(\Omega,\mathcal{F},\p)$, and the coefficient $\mathcal{\sigma}$ is generally nonlinear and satisfies suitable growth assumptions (see Section \ref{section 2} for the complete list of assumptions). In particular, the map $ \vu \mapsto \mathcal{\sigma}(\vu)$ is  a Hilbert space valued function signifying the \emph{multiplicative} nature of the noise.
\subsection{Euler equations}
The Euler equations are the classical model for the motion of an incompressible, inviscid, homogenous fluid. The addition of stochastic terms to the governing equations is commonly used to account for empirical, numerical, and phsyical uncertainties in applications ranging from climatology to turbulence theory. 

\noindent 

In deterministic set-up, for general initial data, the global existence of smooth solution remains a well known open problem for the Euler equations and also their dissipative counterpart, the Naiver-stokes equations. Non-uniqueness was shown for the first time by Scheffer \cite{Sch} who constructed a non trivial weak solution of the 2D incompressible Euler equations with compact support in time. De Lellis, Sz\'ekelyhidi \cite{DelSze3,DelSze} and Chiodaroli et al. \cite{chio01} established ground breaking results, that confirms infinitely many weak solutions can be constructed for the Euler equations in three dimensions. In these works, the method of convex integration was used in order to prove non-uniqueness of weak
solutions to Euler equations. Furthermore, the non-uniqueness was established among weak solutions with dissipating energy which is one of the well-accepted criteria for the selection of physically relevant solutions. In quest for a global-in-time solution, Diperna \cite{DL} proposed a new-concept of solution, known as measure-valued solution to non-linear system of partial differential equation admitting uncontrollable oscillations. Brenier et al. \cite{brenier} proposed a new approach seeing the measure valued solutions as possibly the largest class in which, the family of smooth solutions is stable. In particular, they show the so-called weak (measure-valued)-strong uniqueness principle for the incompressible Euler equations. More specifically, a classical  and a measure-valued solution emanating from the same initial data coincide as long as the former exists. Following the philosophy of Brenier et al. \cite{brenier}, we focus on the concept of measure valued solution in the widest possible-sense.

\noindent

 In stochastic set-up, Glatt-Holtz and Vicol \cite{GHVic} obtained local well-posedness of strong solutions to the
stochastic incompressible Euler equations in two and three dimensions, global well-posedness in two dimensions for
additive and linear multiplicative noise. Local well-posedness for the three dimensional stochastic
compressible Euler equations was proved by Breit and Mensah \cite{BreitM}.
 The convex integration has already been applied in a stochastic setting, namely, to the isentropic Euler system by Breit, Feireisl and Hofmanova \cite{BreitFM} and to the full Euler system by Chiodaroli, Feireisl and Flandoli \cite{EEF}.
There has been a number of attempts to define a suitable notion of measure-valued solutions for the stochastic incompressible Euler equations driven by \emph{additive} noise, starting from the work of Kim \cite{Kim2}, Breit $\&$ Moyo \cite{BrMo}, and most recently by Hofmanova et al. \cite{hof}, where the authors introduced a class of dissipative solutions which allowed them to demonstrate weak-strong uniqueness property and non-uniqueness of solutions in law. However, none of the above mentioned framework can be applied to \eqref{P1}, since the driving noise is \emph{multiplicative} in nature. We also mention recent work \cite{K1, K2, MKS01} on the Euler equations driven by a \emph{multiplicative noise}.
\subsection{Spectral method}
The prototype of spectral methods for the solution of differential equations is the well-known Fourier method which consists of representing the solution as a truncated series expansion, the unknowns being the expansion coefficients. The Fourier basis is appropriate for periodic problems.
Spectral methods have emerged as powerful computional technique for simulation of complex, smooth physical phenomena. Among other applications, they have contributed to our understanding of turbulence by succesfully simulating incompressible turbulent flows which have been extensively used in meteorology, geophysice and have been recentely applied to time domain electromagnetic field (see \cite{Roger}). Spectral method may be viewed as an extreme development of the class of discretization schemes for non-linear differential equations. We refer to \cite{1Tadmor, 2Schochet,3Tadmor,4Tadmor,5Tadmor} for spectral method related articles.

\noindent

We also mention the work of Eitan Tadmor \cite{3Tadmor} in which he discusses behaviour and convergence of Fourier methods for scalar nonlinear conservation laws that exhibit spontaneous shock discontinuities. Mishra et al. \cite{Mishra} combined the spectral (viscosity) method and ensemble averaging to propose an algorithm that computes admissible measure valued solutions of the incompressible Euler equations.
\subsection{Aim and scope of this paper}
In view of wide usage of stochastic fluid dynamics, there is an essential need to improve the mathematical foundations of the stochastic partial differential equations of fluid flow, and in particulary to study inviscid models such as the stochastic incompressible Euler equations. Spectral methods based on projecting into a finite number of Fourier modes are widely employed particularly in the simulation of flows with periodic boundary conditions, while finite difference and finite element methods are very useful when discretizing the Euler equations in domain with complex geometry, there are few results available crossponding to the stochastic incompressible flow equations. To that context, we mention the work of Brz\'ezniak \cite{Brzeniak} where the author study finite-element-based space-time discretizations of the incompressible Navier–Stokes equations with noise. In \cite{K2}, the authors proved existence of measure-valued solutions by showing that weak martingle solutions of the stochastic Naiver-Stokes equations converge to a measure-valued solution of \eqref{P1} as the viscosity tends to zero. But in this work, formulation of measure-valued solutions is slightly different from the given formulation in \cite{K2} (see Definition \ref{def:dissMartin}). Note that our work bears some similarities with recent work of Mishra et al. \cite{Mishra} on deterministic system of the Euler equations. In comparison to previous work of authors \cite{K2} the main novelty of this work lies in successfully handling the multiplicative noise term. Our problems need to invoke ideas from spectral methods for deterministic problems and meaningfully fuse them with available approximation methods for SDEs. In the realm of stochastic conservation laws, noise-noise interaction terms play a fundamental role to establish well-posedness theory, for details see \cite{BhKoleyVa, BhKoleyVa_01, BisKoleyMaj, Koley1, Koley2,koley2013multilevel,Koley3}.
The main contributions of this article are as follows;
\begin{itemize}
\item We study convergence of the spectral method for the incompressible Euler equations driven by a {\em multiplicative noise}. The Cauchy problem for the Euler equations is in general ill-posed in the class of admissible weak solutions. This suggests there might be sequences of approximate solutions that develop fine scale oscillations. Accordingly, the concept of measure-valued solution that capture possible oscillations is more suitable for analysis. We show that sequence of approximate solutions converges to a {\em dissipative measure-valued martingale} solution to the stochastic Euler equations. 
\item In view of the new framework based on the theory of measure-valued solutions, we adapt the concept of $\mathcal{K}$-convergence, first developed in the context of Young measures by Balder \cite{Balder} (see also Feireisl et. al. \cite{FLM1}), to show the pointwise convergence of arithmetic averages (Cesaro means) of  approximate solutions to a \emph{dissipative} solution of the limit system \eqref{P1}.
\item We show that {\em dissipative measure--valued martingale} solutions satisfy a weak–strong uniqueness principle. More precisely, if for some initial data there is an analytically strong solution (defined up to a stopping time), then it coincides with all dissipative measure--valued martingale solutions having the same initial data.
\item When solutions of the stochastic incompressible Euler system possess maximal regularity, by making use of weak (measure-valued)--strong uniqueness principle, we show \emph{unconditional} strong $L^1$-convergence of approximate solutions to the regular solution of the limit system.
\end{itemize}
The paper is organized as follows. In Section \ref{Section 2}, we first introduce our mathematical setting, assumptions and preliminary result. Then, we introduce the definition of {\em dissipative measure-valued martingale} solutions to the incompressible Euler system driven by a \emph{multiplicative noise}, keeping in mind that this framework would allow us to establish weak
(measure-valued)–strong uniqueness principle, and state  the main results of this article in Section \ref{Section 3}. In Section \ref{Section 4}, we give details of the spectral viscosity method to approximate the stochastic incompressible Euler equations. In Section \ref{Section 5}, we prove the convergence of the spectral method in which we present a proof of convergence of approximate solutions to a dissipative measure-valued martingale solutions using stochastic compactness. In Section \ref{Section 6}, we use the concept of $\mathcal{K}$-convergence to exhibit the pointwise convergence of approximate solutions. Section \ref{Section 7} is devoted to deriving the weak (measure-valued)-strong uniqueness principle by making use of a suitable relative energy inequality. Finally, in Section \ref{Section 8}, we make use of weak (measure-valued)-strong uniqueness property to show the convergence of  approximations to the  regular solution of the stochastic incompressible Euler system \eqref{P1}.

\section{Mathematical setting}\label{Section 2}
	\textbf{Function space:}  Let $C_{\text{div}}^\infty(\mathbb{T}^3)$ be the space of infinitely differentiable $3$-dimensional vector fields $u$ on $\mathbb{T}^3$, satisfying $\Div \,u(x)=0$.
$$C_{\text{div}}^\infty(\mathbb{T}^3)=\{\bm{\varphi}\in C^{\infty}(\mathbb{T}^3):\nabla\cdot\bm{\varphi}=0\},$$
$$L_{\text{div}}^2(\mathbb{T}^3)=\bm{cl}_{L^2(\mathbb{T}^3)}C_{\text{div}}^\infty(\mathbb{T}^3)=\{\bm{\varphi}\in L^{2}(\mathbb{T}^3):\nabla\cdot\bm{\varphi}=0\},$$
\noindent
\textbf{Helmholtz projection:} An important consequence of elliptic theory is the existence of the Helmholtz decomposition. It allows to decompose any vector-valued function in $L^2(\mathbb{T}^3;\R^3)$ into a divergence free part and a gradient part. Set
$$(L^2_{\dv}(\mathbb{T}^3))^\perp:=\{\mathbf{u}\in L^2(\T^3;\R^3)|\mathbf{u}=\nabla \psi,\, \psi\in H^{1}(\T^3;\R)   \}$$ 
The Helmholtz decomposition is defined by
$$\mathbf{u}=\mathcal{P}_H \mathbf{u}+\mathcal{Q}_H\mathbf{u}\,\,\,\,\text{for any }\,\,\,\mathbf{u}\in L^2(\mathbb{T}^3)$$
where $\mathcal{P}_H$ is the projection from $L^2(\mathbb{T}^3)$ to $L^2_{\dv}(\mathbb{T}^3)$ and 
$\mathcal{Q}_H=\mathbb{I}-\mathcal{P}_H$ is also projection from $L^2(\mathbb{T}^3)$ to $(L^2_{\dv}(\mathbb{T}^3))^\perp$.
$L^2(\mathbb{T}^3)$ admits a decomposition
$$L^2(\mathbb{T}^3)=L^2_{\dv}(\mathbb{T}^3){\oplus}(L^2_{\dv}(\mathbb{T}^3))^\perp$$ This decomposition is orthogonal with respect to $L^2(\mathbb{T}^3)$-inner product.
By property of projection $\mathcal{P}_H$, we have for $\mathbf{u}\in L^2(\mathbb{T}^3)$
\begin{align}\label{property1}\langle \mathcal{P}_H\mathbf{u},\psi\rangle=\langle \mathbf{u}, \psi \rangle\qquad \text{for all}\,\, \psi\in L_{\text{div}}^2(\mathbb{T}^3). 
\end{align}
\subsection{Stochastic framework}\label{section 2}
Here we specify details of the stochastic forcing term.

\noindent
	\textbf{Brownian motions:} 
	Let $(\Omega,\mathcal{F},(\mathcal{F}_t)_{t\geq0},\prst)$ be a stochastic basis with a complete, right-continuous filtration. The stochastic process $W$ is a cylindrical $(\mathcal{F}_t)$-Wiener process in a separable Hilbert space $\mathfrak{U}$. It is formally given by the expansion
	$$W(t)=\sum_{k\geq 1} e_k W_k(t),$$
	where $\{ W_k \}_{k \geq 1}$ is a sequence of mutually independent real-valued Brownian motions relative to $(\mathcal{F}_t)_{t\geq0}$ and $\{e_k\}_{k\geq 1}$ is an orthonormal basis of $\mathfrak{U}$. Finally, we define the auxiliary space $\mathfrak{U}_0\supset\mathfrak{U}$ via
	$$\mathfrak{U}_0:=\bigg\{\vu=\sum_{k\geq1}\beta_k e_k;\;\sum_{k\geq1}\frac{\beta_k^2}{k^2}<\infty\bigg\},$$
	endowed with the norm
	$$\|\vu\|^2_{\mathfrak{U}_0}=\sum_{k\geq1}\frac{\beta_k^2}{k^2},\quad \vu=\sum_{k\geq1}\beta_k e_k.$$
	Note that the embedding $\mathfrak{U}\hookrightarrow\mathfrak{U}_0$ is Hilbert--Schmidt. Moreover, $\prst$-a.s., trajectories of $W$ are in $C([0,T];\mathfrak{U}_0)$.
	
	\noindent
	\textbf{Multiplicative noise:}
	To give the precise definition of the diffusion coefficient $\mathcal{\sigma}$, consider $\mathbf{u}:\Omega\times[0,T]\time\Omega\to L^2(\mt;\R^3)$ is predictable . 
	Let $\,\mathcal{\sigma}({\bf u}):\mathfrak{U}\rightarrow L^2(\mt;\R^3)$ be defined as follows
	$$\mathcal{\sigma}({\bf u})e_k=\mathcal{\sigma}_k({\bf u}(\cdot)).$$
	The coefficients $\mathcal{\sigma}_{k}:\mr^3\rightarrow\mr^3$ are $C^1$-functions that satisfy
	\begin{align}
		\sum_{k\ge1}|\mathcal{\sigma}_k({\bf u})|^2 \le D_0 (1+|{\bf u}|^2),\label{FG1}\\
		\sum_{k\ge1}|\mathcal{\sigma}_k({\bf u})-\mathcal{\sigma}_k({\bf v})|^2\le D_1 |{\bf u}-\textbf{v}|^2.
		\label{FG2}
	\end{align}
	
	Then the stochastic integral
	\[
	\int_0^t \mathcal{\sigma}(\vu) \ {\rm d} W = \sum_{k \geq 1}\int_0^t \mathcal{\sigma}_k (\vu) \ {\rm d} W_k
	\]
	is a well-defined $(\mathcal{F}_t)$-martingale.
\subsection{Preliminary result}
\textbf{Modified version of Jakubowski-Skorokhod theorem:}
 Strong convergence of approximate solutions in $\omega$ variable plays a pivotal role. To that context, we need Jakubowski-Skorokhod theorem, delivering a new probability space and new random variables, with the same laws as the original ones, converging almost surely.
 However, for technical reasons, we have to use a modified version of Jakubowski-Skorokhod theorem \cite[Corollary 7.3]{Motyl} which is stated below.
\begin{Theorem}\label{newth}
	Let $(\Omega, \mathcal{F},\mathbb{P})$ be a probability space and $S_1$ be separable metric space and $S_2$ be a quasi-polish space (there is a sequence of continuous functions $h_n:S_2\to[-1,1]$ that separates points of $S_2$). $\mathcal{B}(S_1)\otimes\mathcal{S}_2$ is sigma algebra associated with product space $S_1\times S_2$, where $\mathcal{S}_2$ is the sigma algebra generated by the sequence of $h_n$. Let $U_n:\Omega\to S_1\times S_2$, $n\in\mathbb{N}$, be a family of random variables, such that the sequence $\{\mathcal{L}aw(U_n):n\in\mathbb{N}\}$ is weakly convergent on $S_1\times S_2$.
	For $k=1,2,$ let $\pi_i:S_1\times S_2$ be the projection onto $S_i$, i.e.
	$$U=(U_1,U_2)\in S_1\times S_2\mapsto\pi_i(U)=U_i\in S_i.$$ 
	Finally let us assume that there exists a random variable $X:\Omega\to S_1$ such that $\mathcal{L}aw(\pi (U_n))=\mathcal{L}aw(X),\,\forall\,n\in\mathbb{N}$.
	Then, there exists a probability space $(\widetilde{\Omega},\widetilde{\mathcal{F}},\widetilde{\mathbb{P}})$, a family of $S_1\times S_2$-valued random variables $\{\widetilde{U}_n:n\in\mathbb{N}\}$, on $(\widetilde{\Omega},\widetilde{\mathcal{F}},\widetilde{\mathbb{P}})$ and a random variable $\widetilde{U}:\widetilde{\Omega}\to S_1\times S_2$ such that
	\begin{enumerate}
		\item  $\mathcal{L}aw(\widetilde{U}_n)=\mathcal{L}aw(U_n)\,\forall\,n\in\mathbb{N};$
		\item
		$\widetilde{U}_n\to \widetilde{U}\,\text{in}\,S_1\times S_2,\,\mathbb{P}-\,\text{a.s.}$
		\item
		$\pi_1(\widetilde{U}_n)(\widetilde{w})=\pi_1(\widetilde{U})(\widetilde{w}),\,\forall\,\widetilde{w}\in\widetilde{\Omega}. $
	\end{enumerate}
\end{Theorem}
\section{Definitions and main results}\label{Section 3}
\subsection{Dissipative measure-valued martingale solutions }
 We are ready to introduce the concept of \textit{dissipative measure--valued martingale solution} to the stochastic incompressible Euler system. In what follows, let $\mathcal{M} =  \R^3$ be the phase space associated to the incompressible Euler system. 
\begin{Definition}[Dissipative measure-valued martingale solution]
	\label{def:dissMartin}
	Let $\Lambda$ be a Borel probability measure on $L_{\text{div}}^2(\mathbb{T}^3)$. Then $\big[ \big(\Omega,\mathcal{F}, (\mathcal{F}_{t})_{t\geq0},\mathbb{P} \big); \mathcal{V}^{\omega}_{t,x}, W \big]$ is a dissipative measure-valued martingale solution of \eqref{P1}, with initial condition $\mathcal{V}^{\omega}_{0,x}$; if
	\begin{enumerate}[(a)]
		\item $\mathcal{V}^{\omega}$ is a random variable taking values in the space of Young measures on $L^{\infty}_{w^*}\big([0,T] \times \T^3; \mathcal{P}\big(\mathcal{M})\big)$. In other words, $\p$-a.s.
		$\mathcal{V}^{\omega}_{t,x}: (t,x) \in [0,T] \times \T^3  \rightarrow \mathcal{P}(\mathcal{M})$ is a parametrized family of probability measures on $\mathcal{M}$,
		\item $ \big(\Omega,\mathcal{F}, (\mathcal{F}_{t})_{t\geq0},\mathbb{P} \big)$ is a stochastic basis with a complete right-continuous filtration,
		\item $W$ is a $(\mathcal{F}_{t})$-cylindrical Wiener process,
		\item the average velocity $\langle \mathcal{V}^{\omega}_{t,x}; \vu  \rangle$ satisfies,  for any $\bm{\varphi}\in C_{\text{div}}^\infty(\mathbb{T}^3)$,  $t\mapsto \langle \langle \mathcal{V}^{\omega}_{t,x}; \vu  \rangle (t, \cdot),\bm\varphi\rangle\in C[0,T]$, $\mathbb{P}$-a.s., the function $t\mapsto \langle \langle \mathcal{V}^{\omega}_{t,x}; \vu  \rangle (t, \cdot),\bm{\phi}\rangle$ is progressively measurable, and for any $\bm\varphi\in C^1(\mathbb{T}^3)$,
		\begin{align*}
			\int_{\mathbb{T}^3}\langle \mathcal{V}^{\omega}_{t,x}; \vu  \rangle\cdot\nabla_x \bm{\varphi}\,dx=0
		\end{align*}
		for all $t\in[0,T]$, $\mathbb{P}-$a.s., and 
		\begin{align*}
			\mathbb{E}\, \bigg[ \sup_{t\in(0,T)}\Vert  \langle \mathcal{V}^{\omega}_{t,x}; \vu  \rangle (t, \cdot) \Vert_{L_{\dv}^2(\mt)}^p\bigg]<\infty 
		\end{align*}
		for all $1\leq p<\infty$,
		\item $\Lambda=\mathcal{L}[\mathcal{V}^{\omega}_{0,x}]$,
		\item  the integral identity 
		\begin{equation} \label{second condition measure-valued solution}
			\begin{aligned}
				&\int_{\T^3} \langle \mathcal{V}^{\omega}_{\tau,x};\vu \rangle \cdot \bm{\varphi} \,\D x - \int_{\T^3} \langle \mathcal{V}^{\omega}_{0,x};\vu \rangle \cdot \bm{\varphi}\, \D x \\
				&\qquad = \int_{0}^{\tau} \int_{\T^3}\langle \mathcal{V}^{\omega}_{t,x}; {\vu \otimes \vu }\rangle: \nabla_x \bm{\varphi} \, \D x \, \D t + \int_0^{\tau}\,\int_{\T^3} \langle \mathcal{V}^{\omega}_{\tau,x};\mathcal{\sigma}(\vu)\rangle\cdot\bm{\varphi}\, \D W(t) \,\D x+ \int_{0}^{\tau} \int_{\T^3}  \nabla_x \bm{\varphi}: d\lambda_\mathcal{C},
			\end{aligned}
		\end{equation}
		holds  $\p$-a.s., for all $\tau \in [0,T)$, and for all $\bm{\varphi} \in C_{\text{div}}^{\infty}(\T^3;\mathbb{R}^3)$, where $\lambda_\mathcal{C}\in L_{w*}^\infty([0,T];\mathcal{M}_b({\T^3}))$, $\p$-a.s., is a tensor--valued measure; $\lambda_\mathcal{C}$ is called \textit{concentration defect measures};
		\item there exists a real-valued square integrable continuous martingale $\mathcal{M}^2_{E}$, such that the following inequality
		\begin{align} \label{third condition measure-valued solution}
			\begin{aligned}
				\mathcal{E}(t+)\, \leq\, &\mathcal{E}(s-)+ \frac{1}{2} \sum_{k\ge\,1}\int_s^{t} \int_{\T^3} \left\langle \mathcal{V}^{\omega}_{\tau,x};{|{\mathcal{\sigma}_k(\vu)}|^2} \right\rangle \,\, \D x\, {\rm d}\tau \\
				&\quad  -\frac{1}{2}\sum_{k\ge\,1}\int_s^t\int_{\mathbb{T}^3}\Big(\mathcal{Q}_H\left\langle \mathcal{V}^{\omega}_{\tau,x};{|{\mathcal{\sigma}_k(\vu)}|}\right\rangle\Big)^2\,\D x\,{\rm d}\tau + \frac12\int_s^{t} \int_{\T^3} d \lambda_{\mathcal{D}} + \int_s^{t}  d\mathcal{M}^2_{E},
			\end{aligned}			
		\end{align}
		holds  $\p$-a.s., for all $0\le\,s\,\textless\,t\in (0,T)$ with
	$$\mathcal{E}(t-):=\liminf_{r\to 0^+}\frac{1}{r}\int_{t-r}^t\bigg(\int_{\mathbb{T}^3}\left\langle \mathcal{V}^{\omega}_{s,x};\frac{|{\bf u}|^2}{2} \right\rangle \dx +\mathcal{H}(s)\bigg)\,{\rm d}s$$
		$$\mathcal{E}(t+):=\liminf_{r\to 0^+}\frac{1}{r}\int_t^{t+r}\bigg(\int_{\mathbb{T}^3}\left\langle \mathcal{V}^{\omega}_{s,x};\frac{|{\bf u}|^2}{2} \right\rangle \dx +\mathcal{H}(s)\bigg)\,{\rm d}s$$
		Here $\lambda_\mathcal{D} \in L_{w*}^\infty([0,T];\mathcal{M}_b({\T^3}))$, $\p$-a.s., $\mathcal{H}\in L^{\infty}(0,T)$, $\mathcal{H}\geq 0$, $\p$-almost surely, and 
		$$\mathbb{E} \big[ \sup_{t \in (0,T)}\mathcal{H}(t)\big] < \infty,$$ with initial energy
		$$
		\mathcal{E}(0-)= 
		\int_{\mathbb{T}^3} \frac{1}{2} |\textbf{u}_0|^2\,\dx,
		$$
		\item there exists a constant $C>0$ such that
		\begin{equation} \label{fourth condition measure-valued solutions}
			\int_{0}^{\tau} \int_{\T^3} d|\lambda_\mathcal{C}| + \int_{0}^{\tau} \int_{\T^3} d|\lambda_\mathcal{D}| \leq C \int_{0}^{\tau} \mathcal{H}(t) dt,
		\end{equation}	
		$\p$-a.s., for every $\tau \in (0,T)$.	
	\end{enumerate}
\end{Definition}	
\begin{Remark}
	Notice that, a standard Lebesgue point argument applied to \eqref{third condition measure-valued solution} reveals that the energy inequality holds for a.e. $0 \le s <t$ in $(0,T)$:
	\begin{align}
		\label{energy_001}
		& \int_{\mathbb{T}^3}\left\langle \mathcal{V}^{\omega}_{t,x};\frac{|{\bf u}|^2}{2} \right\rangle \dx +\mathcal{H}(t) \, \leq\, \int_{\mathbb{T}^3}\left\langle \mathcal{V}^{\omega}_{s,x};\frac{|{\bf u}|^2}{2} \right\rangle \dx +\mathcal{H}(s) + \frac{1}{2} \sum_{k\ge\,1}\int_s^{t} \int_{\T^3} \left\langle \mathcal{V}^{\omega}_{\tau,x};{|{\sigma_k( \mathbf{u})}|^2} \right\rangle \,\, \D x\, {\rm d}\tau \notag \\ &\quad  -\frac{1}{2}\sum_{k\ge\,1}\int_s^t\int_{\mathbb{T}^3}\Big(\mathcal{Q}_H\left\langle \mathcal{V}^{\omega}_{\tau,x};{|{\sigma_k(\mathbf{ u})}|}\right\rangle\Big)^2\,\D x\,{\rm d}\tau + \frac12\int_s^{t} \int_{\T^3} d \lambda_D + \int_s^{t}  d\mathcal{M}^2_{E}, \,\,\p-a.s.
	\end{align}
	However, as it is evident from Section~\ref{Section 7}, we require energy inequality to hold for \emph{all} $s, t \in (0,T)$ to demonstrate weak-strong uniqueness principle. 
\end{Remark}
\subsection{Strong pathwise solutions.}
We are also interested in establishing weak (measure-valued)--strong uniqueness principle for dissipative measure-valued solutions to \eqref{P1}. Since such an argument requires the existence of strong solution, therefore, we first recall the notion of local strong pathwise solution for the stochastic incompressible Euler equations. We remark that such a solution can be constructed on any given stochastic basis, that is, solutions are probabilistically strong, and satisfies the underlying equation \eqref{P1} pointwise (not only in the sense of distributions), that is, solutions are srtong from the PDE standpoint. Existence of such a solution was first established by Glatt-Holtz $\&$ Vicol in \cite{GHVic}.
 
\begin{Definition}\textbf{(Local strong pathwise solution).} 
	\label{def:strsol}
	Let $(\Omega,\mathcal{F},(\mathcal{F}_t)_{t\geq0},\prst)$ be a stochastic basis with a complete right-continuous filtration, and ${W}$ be an $(\mathcal{F}_t) $-cylindrical Wiener process. Let $\vu_0$ be a $W^{m,p}(\T^3)$-valued $\mathcal{F}_0$-measurable random variable. Then
	$(\vu,\mathfrak{t})$ is said to be a local strong pathwise solution to the system \eqref{P1} provided
	\begin{enumerate}[(a)]
		\item $\mathfrak{t}$ is an a.s. strictly positive  $(\mathcal{F}_t)$-stopping time;
		\item the velocity $\vu$ is a $W_{\dv}^{m,p}(\mt)$-valued $(\mathcal{F}_t)$-predictable measurable process satisfying
		$$ \vu(\cdot\wedge \mathfrak{t}) \in   C([0,T]; W^{m,p}_{\Div}(\mt))\quad \mathbb{P}\text{-a.s.};$$
		\item for all $t\,\ge\,0$,
		\begin{equation}
			\begin{split}
				\vu (t \wedge \mathfrak{t})&= \vu_0- \int_0^{t \wedge \mathfrak{t}}\mathcal{P}_H(\vu \cdot\nabla\vu )\dif s + \int_0^{t \wedge \mathfrak{t}}\mathcal{P}_H{\mathcal{\sigma}}(\vu ) \ \D W.
			\end{split}
		\end{equation}
	\end{enumerate}
\end{Definition}
It is evident that classical solutions require spatial derivatives of the velocity field $\vu$ to be continuous $\prst$-a.s. This motivates the following definition.

\begin{Definition}\textbf{(Maximal strong pathwise solution).}\label{def:maxsol}
	Fix an initial condition, and a complete stochastic basis with a cylindrical Wiener process as in Definition \ref{def:strsol}. Then a triplet $$(\vu,(\tau_R)_{R\in\mn},\mathfrak{t})$$ is said to be a maximal strong pathwise solution to system \eqref{P1} provided
	
	\begin{enumerate}[(a)]
		\item $\mathfrak{t}$ is an a.s. strictly positive $(\mathbb{F}_t)$-stopping time;
		\item $(\tau_R)_{R\in\mn}$ is an increasing sequence of $(\mathbb{F}_t)$-stopping times such that
		$\tau_R<\mathfrak{t}$ on the set $[\mathfrak{t}<T]$,
		$\displaystyle{\lim_{R\to\infty}\tau_R}=\mathfrak t$ a.s. and
		\begin{equation}\label{eq:blowup}
			\sup_{t\in[0,\tau_R]}\|\vu(t)\|_{W^{1,\infty}(\T^3)}\geq R\quad \text{on}\quad [\mathfrak{t}<T] ;
		\end{equation}
		\item each pair $(\vu,\tau_R)$, $R\in\mn$,  is a local strong pathwise solution in the sense  of Definition \ref{def:strsol}.
	\end{enumerate}
\end{Definition}
\subsection{Statements of main results}

We now state main results of this paper. To begin with, regarding the existence of dissipative measure-valued martingale solutions, we have the following result.
\begin{Theorem}[\textbf{Existence of measure-valued solutions}] Let $\vu_0\in L_{\dv}^2(\T^3)$.
	 Then approximating solutions $\vu_n$ resulted by the spectral viscosity method \eqref{approximate equations} (semi-discrete scheme)  generates a dissipative measure-valued martingale solution in the sense of Definition \ref{def:dissMartin} to the incompressible Euler system \eqref{P1}.
\end{Theorem}
Next, we make use of the $\mathcal{K}$-convergence in the context of Young measures to conclude the following pointwise convergence of averages of approximate solutions to a dissipative martingale solution to \eqref{P1}.
\begin{Theorem}[\textbf{Point-wise convergence to a dissipative solution}]
	\label{dissipative solution}Suppose that the approximate solutions $\vu_n $ to \eqref{approximate equations} for the stochastic Euler system generates a dissipative measure-valued martingale solution $\big[ \big(\widetilde{\Omega},\widetilde{\mathcal{F}}, (\widetilde{\mathcal{F}}_{t})_{t\geq0},\widetilde{\mathbb{P}} \big); \widetilde{\mathcal{V}}^{\omega}_{t,x}, \widetilde{W} \big]$ in the sense of Definition~\ref{def:dissMartin}. Then there exists a sequence of approximate solutions $\widetilde{\vu}_n$ to \eqref{approximate equations} on probability space $(\widetilde{\Omega},\widetilde{\mathcal{F}}, (\widetilde{\mathcal{F}}_{t})_{t\geq0},\widetilde{\mathbb{P}} \big)$ for which following holds true,
	\item[1.]  $\widetilde{\p}$-a.s.
	$$\widetilde{\vu}_{n}\to\langle {\widetilde{\mathcal{V}}^{\omega}_{t,x}}; {\widetilde{\vu}} \rangle\,\,\mbox{in}\,\,\,C_w([0,T],L_{\dv}^{2}(\mathbb{T}^3)),$$ 
	\item[2.] $\widetilde{\p}$-a.s., there exists subsequence $\widetilde{\vu}_{n_k}$ such that
	\begin{equation*}
		\begin{aligned}
			\frac 1N \sum_{k=1}^N \widetilde{\vu}_{n_k} &\to\langle {\widetilde{\mathcal{V}}^{\omega}_{t,x}}; \widetilde{\vu} \rangle, \ \mbox{as $N \rightarrow \infty$ a.e. in} \,\,(0,T)\times\T^3.	
		\end{aligned}
	\end{equation*}		
\end{Theorem}
\begin{Theorem}[\textbf{Weak-strong uniqueness}] \label{Weak-Strong Uniqueness_01}
	Let $\big[ \big(\Omega,\mathcal{F}, (\mathcal{F}_{t})_{t\geq0},\mathbb{P} \big); \mathcal{V}^{\omega}_{t,x}, W \big]$ be a dissipative measure-valued martingale solution to the system \eqref{P1}. On the same stochastic basis $\big(\Omega,\mathcal{F}, (\mathcal{F}_{t})_{t\geq0},\mathbb{P} \big)$, let us consider the unique maximal strong pathwise solution in sense of Definition \ref{def:maxsol} to the Euler system \eqref{P1} given by 
	$(\bar{{\bf u}},(\mathfrak{t}_R)_{R\in\mn},\mathfrak{t})$ driven by the same cylindrical Wiener process $W$ with the initial data $\bar{\bf u}(0)$ satisfies 
	\begin{equation*} 
		\mathcal{V}^{\omega}_{0,x}= \delta_{\bar{\bf u}(0,x)}, \,\p-\mbox{a.s.,}\, \mbox{for a.e. } x \in \T^3.
	\end{equation*}
	Then, $\p$-a.s. a.e. $t\in[0,T]$ $\mathcal{H}(t\wedge\mathfrak{t}_R)=0$, and $\p-\mbox{a.s.,}$
	\begin{equation*}
		\mathcal{V}^{\omega}_{t \wedge \mathfrak{t}_R,x}	= \delta_{\bar{\bf u}(t \wedge \mathfrak{t}_R,x)}, \,\,  \mbox{for a.e. }(t,x)\in (0,T)\times \T^3.
	\end{equation*}
\end{Theorem} 

\noindent Finally, making use of the weak (measure-valued)--strong uniqueness principle (cf. Theorem~\ref{Weak-Strong Uniqueness_01}), we prove the following result justifying the strong convergence to the regular solution.
\begin{Theorem}[\textbf{Strong convergence to regular solution}]
\label{T_ccE}
Let $u_0\in L_{\dv}^2(\T^3)$. Suppose that the approximate solutions $\vu_n$ to \eqref{approximate equations} for the stochastic Euler system generate a dissipative measure-valued martingale solution $\big[ \big(\widetilde{\Omega},\widetilde{\mathcal{F}}, (\widetilde{\mathcal{F}}_{t})_{t\geq0},\widetilde{\mathbb{P}} \big); \widetilde{\mathcal{V}}^{\omega}_{t,x}, \widetilde{W} \big]$  in the sense of Definition~\ref{def:dissMartin}.
In addition, let the Euler equations \eqref{P1} possess the unique strong (continuously differentiable) solution $(\bar{\vu}, (\mathfrak{t}_R)_{R\in\mn}, \mathfrak{t})=(\bar{\vu}, (\mathfrak{t}_R)_{R\in\mn}, \mathfrak{t})$, emanating form the initial data \eqref{P1}. Then there exists a sequence of approximate solutions $\widetilde{\vu}_n$ to \eqref{approximate equations} on probability space $(\widetilde{\Omega},\widetilde{\mathcal{F}}, (\widetilde{\mathcal{F}}_{t})_{t\geq0},\widetilde{\mathbb{P}} \big)$  such that $\widetilde{\p}$-a.s.
\begin{equation*}
	\begin{aligned}
		\widetilde{\vu}_{n}(\cdot\wedge\mathfrak{t}_R) &\to \bar{\vu}(\cdot\wedge\mathfrak{t}_R) \ \mbox{weakly-(*) in} \ L^{\infty}(0,T;L_{\dv}^2(\T^3))
		\ \mbox{and strongly in}\ L^1((0,T) \times \T^3).
		\end{aligned}
	\end{equation*}
\end{Theorem}

\section{Fourier method for the incompressible Euler system}\label{Section 4}
We demonstrate the abstract theory applying the results to the approximate solutions resulting from Fourier approximation of the Euler system.
\begin{itemize}
	\item \textbf{Existence of approximate solutions.} First we recall the existence of the Fourier approximate solutions $\vu_n$ of semi-discrete scheme in Fourier mode for any discretization $n\in\mathbb{N}$. We introduce here a kind of spectrally accurate vanishing viscosity to augment the Fourier approximation of such nonlinear equations. 
	\item \textbf{Stability and a priori bounds.} We assure that the scheme is energy dissipative. We recover required energy bounds from energy inequality.
	\item \textbf{Consistency.} We provide a consistency formulation and establish suitable bounds on the error terms.
	\item \textbf{Convergence of spectral method.} Using the stochastic compactness technique we show that approximate solutions generate a {\em dissipative measure valued martingale} solution. The proof relies on a compactness argument combined with Jakubowski–Skorokhod’s representation theorem. Due to the limited compactness of the Euler system it is necessary to work with dissipative rather than analytically weak
	solutions.
%
\end{itemize}
\subsection{Preliminaries for spectral method}
We begin by reviewing some basic tools associated with spectral method.

\noindent
 \textbf{Fourier Coefficient:} Consider the spatial Fourier expansion $\vu(x,t)=\sum_{k}\widehat{\vu}_k(t)e^{ik\cdot x}$ with coefficients $\widehat{\vu}_k$ given by
$$\widehat{\vu}_{k}(t)=\int_{\T^3}\vu(x,t)e^{-ik\cdot x}\dx,$$

\noindent
\textbf{Truncation Operator:} Truncation Operator $\mathcal{T}_n$ project vector field of the form $\vu=\sum_{k}\widehat{\vu}_k(t)e^{ik\cdot x}$ to  $\sum_{|k|\,\le n}\widehat{\vu}_k(t)e^{ik\cdot x}$ (only Fourier modes below thresold $n$). That is
$$\mathcal{T}_n(\vu)=\sum_{|k|\,\le n}\hat{\vu}_k(t)e^{ik\cdot x}$$

\noindent
\textbf{Projection Operator:} We know Helmholtz projection project vector field $\vu=\sum_{k}\widehat{\vu}_k(t)e^{ik\cdot x}$ to divergence-free vector field given by
$$\mathcal{P}_H(\vu)=\sum_{k}\big(\widehat{\vu}_k-\frac{\widehat{\vu}_k\cdot k}{|k|^2}\big)e^{ik\cdot x}$$
Here we consider finite trunction of Holmoltz Projection as $\mathcal{P}_N$ given by
\begin{equation*}
	\mathcal{P}_N(\vu)=\mathcal{T}_n\big(\mathcal{P}_H(\vu)\big)=\sum_{|k|\le N}\big(\widehat{\vu}_k-\frac{\widehat{\vu}_k\cdot k}{|k|^2}\big)e^{ik\cdot x}
\end{equation*}
yielding a divergene-free vector field with Fourier modes $|k|\le N$.
We also define 
$$\mathcal{Q}_n(\vu)=\mathcal{T}_n(\vu)-\mathcal{P}_m(\vu)$$
where $m\textless n$ and $\mathcal{Q}_n$ shows the projection onto upper modes.
\subsection{Semi-discrete scheme for spectral method}
We propose a spectral viscosity method (SVM) to approximate the stochastic incompressible Euler
equations and prove that SVM solution converges to a dissipative measure-valued martingale solution.

\noindent
\textbf{Motivation:} To motivate the semi-discrete scheme, let $(\vu,\Pi)$ be solutions to \eqref{P1} with periodic boundary. We focus on the spectral method based on the Fourier expansion and at the heart of a spectral method lies the assumption that the solutions $\vu(x,t)$ can be expressed by a series of smooth basis functions. So we consider the spatial Fourier expansion $\vu(x)=\sum_k \widehat{\vu}_k e^{ik\cdot x}$, $B=\vu\cdot\Grad\vu=\sum_k \widehat{B}_k(\vu) e^{ik\cdot x}$ and ${\mathcal{\sigma}}(u)=\sum_k \widehat{\mathcal{\sigma}}_k(\vu) e^{ik\cdot x}$. Therefore divergence free condition gives that
\begin{equation}\label{P3}
	i\sum_k({\widehat{\vu}_k}\cdot k) e^{ik\cdot x}=0\iff \widehat{\vu}_k\cdot k=0 \,\,\,\forall\,\, k
\end{equation}
In terms of Fourier coefficients we have \eqref{P1} in this form
\begin{equation}\label{P4}
	\D {\widehat{\vu}_k}(t) + \widehat{B}_k(\vu)\dt + ik\widehat{\Pi}_k \dt = \widehat{\mathcal{\sigma}}_k(\vu)\D W 
\end{equation}
Take dot product of \eqref{P4} with k and used \eqref{P3}, therefore
\begin{equation*}
	\widehat{B}_k(\vu)\cdot k\dt+i|k|^2 \widehat{\Pi}_k\dt=\widehat{\mathcal{\sigma}}_k(\vu)\cdot k \D W
\end{equation*}
Eliminate pressure term from \eqref{P4} using above expression, 
\begin{equation}
	\D\widehat{\vu}_k +(\widehat{B}_k(\vu)-\frac{\widehat{B}_k(\vu)\cdot k}{|k|^2}k)\dt=(\widehat{\mathcal{\sigma}}_k(\vu)-\frac{\widehat{\mathcal{\sigma}}_k(\vu)\cdot k}{|k|^2}k)\D W
\end{equation}
For the coefficient $\widehat\vu_k$ with k=0, we can assume that $\int_{\T^3}\vu_0\dx=0.$

\noindent
\textbf{Semi-discrete scheme:} To obtain a semi discretized approximation to system \eqref{P1}. We restrict our attention to only the Fourier modes below some thresold $n$. We consider velocity field of the form $\vu_n=\sum_{|k|\le\,n}\widehat{\vu}_k\,e^{ik\cdot x}$ and we have to project the non-linear term to this space.
 
\noindent In the following, we will consider the spectral vanishing viscosity scheme for the stochastic incompressible Euler equations. Consider the following approximation of the stochastic incompressible Euler Equations
  \begin{equation}\label{approximate equations}
  	\begin{cases} 
  		\D \vu_n+ \mathcal{P}_n(\vu_n\cdot\nabla\vu_n)\dt=\varepsilon\,\dv(\mathcal{Q}_n\nabla \vu_n)\dt + \mathcal{P}_n\mathcal{\sigma}(\vu_n)\D W \\ 
  			\vu_n(0)=\mathcal{T }_n(\vu_0)\\ 
  	\end{cases}
  \end{equation}
  In this scheme, we adopt a small $\varepsilon:=\varepsilon(n)$ $(\varepsilon(n)\to 0\,\, \text{as}\,\, n\to \infty)$ and an integer $m\textless\,n$. Here the integer $m$ handle as a thresold between small and large Fourier modes. We also added a small amount of numerical viscosity to ensure stability of the resulting scheme. The idea behind the SVM is that dissipation is only applied on the upper part of the spectrum $(m\textless n)$. System includes following  Naiver-Stokes system (for example $m=0$).
\begin{equation}
	\begin{cases} 
	\D \vu_n+ \mathcal{P}_n(\vu_n\cdot\nabla\vu_n)\dt=\varepsilon\,\Delta\vu_n \dt + \mathcal{P}_n\mathcal{\sigma}(\vu_n)\D W \\ 
	\vu_n(0)=\mathcal{T}_n(\vu_0)\\ 
	\end{cases}
\end{equation}
\textbf{Existence of approximate solutions $\vu_n$:} The existence of solutions $\vu_n$ to \eqref{approximate equations} is classical and relies on a priori bounds that are established using the cancellation property. For a proof, one can follows the similar approach as proposed in \cite{Flandoli}.
\subsection{Stability and energy bounds}
	\textbf{Energy inequality to $\vu_n$:} We derive the energy inequality from the scheme. The energy inequality is a direct consequence of It\^o formula.
\begin{Lemma}
	Let $\vu_n$ be the solution of the semi-discrete scheme \eqref{approximate equations}. Then, $\p$-a.s., for all $s\,\textless\,t$,
	\begin{align}\label{L0}
		&\frac{1}{2}\|\vu_n(t)\|_{L^2(\T^3)}^2 +  \varepsilon\int_s^t\|\mathcal{Q}_n(\nabla\vu_n(\tau))\|_{L^2(\mathbb{T}^3)}^2\D \tau\,\notag\\
		&=\frac{1}{2}\|\vu_n(s)\|_{L^2(\mathbb{T}^3)}^2 + \int_s^t\int_{\mathbb{T}^3}\vu_n(\tau)\mathcal{P}_n(\mathcal{\sigma}(\vu_n(\tau)))\D W(\tau) + \int_s^t\|\mathcal{P}_n(\vu_n(\tau))\|_{L^2(\T^3)}^2\dx\D \tau.
	\end{align}
In particular, $\p$-a.s., for all $t\in[0,T]$
\begin{align}\label{L}
	&\frac{1}{2}\|\vu_n(t)\|_{L^2(\T^3)}^2 +  \varepsilon\int_0^t\|\mathcal{Q}_n(\nabla\vu_n(\tau))\|_{L^2(\mathbb{T}^3)}^2\D \tau\,\notag\\
	&\le\,\frac{1}{2}\|\vu_0(s)\|_{L^2(\mathbb{T}^3)}^2 + \int_0^t\int_{\mathbb{T}^3}\vu_n(\tau)\mathcal{P}_n(\mathcal{\sigma}(\vu_n(\tau)))\D W(\tau) + \int_0^t\|\mathcal{P}_n(\vu_n(\tau))\|_{L^2(\T^3)}^2\dx\D \tau.
\end{align}
\end{Lemma}
\begin{proof}
	Apply It\^o formula to $F(\vu)=\frac{1}{2}{\|\vu\|_{L^2(\mathbb{T}^3)}^2}$, we get $\p$-a.s., for all $ s\,\textless\,t\in[0,T]$
	\begin{align}\label{L1}
		\frac{1}{2}\|\vu_n(t)\|_{L^2(\T^3)}^2=&\frac{1}{2}\|\vu_n(s)\|_{L^2(\T^3)}^2-\int_{s}^t\int_{\T^3}\big( \vu_n\cdot\mathcal{P}_n(\vu_n\cdot\nabla\vu_n)-\varepsilon\vu_n\cdot\dv(\mathcal{Q}_n\nabla\vu_n)\big)\dx \D \tau\notag\\&+\int_s^t\int_{\T^3}\vu_n\mathcal{P}_n(\mathcal{\sigma}(\vu_n))\D x\D W(\tau)+\int_s^t\|\mathcal{P}_n(\mathcal{\sigma}(\vu_n))\|_{L^2(\T^3)}^2d\tau
		\end{align}
	\begin{align}\label{L2}
\int_{\T^3} \vu_n\cdot\mathcal{P}_n(\vu_n\cdot\nabla\vu_n)\dx=\int_{\T^3}\vu_n\cdot(\vu_n\cdot\nabla\vu_n)\dx=\int_{\T^3}\dv(\frac{1}{2}|\vu_n|^2\vu_n)\dx=0
	\end{align}
\begin{align}\label{L3}
	\varepsilon\int_{\T^3}\vu_n\cdot\dv(\mathcal{Q}_n\nabla\vu_n)\dx=&-\varepsilon\int_{\T^3}\nabla\vu_n:\mathcal{Q}_n\nabla\vu_n \dx=-\varepsilon\int_{\T^3}\mathcal{Q}_n\nabla\vu_n:\mathcal{Q}_n\nabla\vu_n \dx\notag\\
	&=-\varepsilon\int_{\T^3}|\mathcal{Q}_n(\nabla\vu_n)|^2 \dx
\end{align}
Use\eqref{L1}-\eqref{L3} to get \eqref{L0}. For second part \eqref{L}, just use the fact that $\|\mathcal{T}_n(\vu_0)\|_{L^2(\T^3)}^2\le\,\|\vu_0\|_{L^2(\T^3)}^2$. Our claim is done.
\end{proof}

\noindent
\textbf{A priori estimates:} In what follows we derive a priori bounds from energy inequality.
After taking $p-$th power and expectation of both side \eqref{L}, making use of Gronwall's and BDG inequality, we immediately get following uniform bounds in $n$, for all $p\,\ge\,1$,
\begin{align}\label{uniform estimate}
	\mathbb{E}\big[\sup_{t\in[0,T]}\|\vu_n(t)\|_{L^2(\T^3)}^p\big]\le\,\|\vu_0\|_{L^2(\T^3)}^p.
\end{align}
\subsection{Consistency formulation}
In this section, our aim is to prove the consistency of the momentum equation. We demonstrate a consistency formulation and derive suitable bounds on the error terms. The consistency formulation of semi-discrete scheme for the incompressible  Euler equations reads, for all $t\in[0,T]$, $\p$-a.s.
\begin{align}\label{consistency formulation}
	\langle \vu_n(t), \bm\varphi \rangle=\langle\mathcal{T}_n(\vu_0),\bm\varphi\rangle +\int_0^t \langle \vu_n\otimes\vu_n,\nabla\bm\varphi\rangle ds+\int_0^t\langle \mathcal{\sigma}(\vu_n),\bm\varphi\rangle\D W(s)+ \mathcal{R}_1(n,\bm\varphi)+\mathcal{N}(n,\bm\varphi) 
\end{align}
where $\mathcal{R}_1(n,\mathbf{\bm\varphi})$, $\mathcal{N}(n,\bm\varphi)$ satiesfies $\p$-a.s
$$\mathcal{R}_1(n,\bm\varphi):=-\int_{\T^3}\nabla(\mathbb{I}-\mathcal{P}_n)\bm\varphi:(\vu_n\otimes\vu_n)\dx$$
$$\mathcal{N}(n,\bm\varphi):=\varepsilon\int_{\T^3}(\mathbb{I}-\mathcal{P}_m)\Delta\bm\varphi\cdot\vu_n \dx,$$
$$|\mathcal{R}_1(n,\bm\varphi)|\le\,C\,\|\vu_n\|_{L^2(\T^3)}^2\|(\mathbb{I}-\mathcal{P}_n)\bm\varphi\|_{H^{3/2}(\T^3)},$$
$$	|\mathcal{N}(n,\bm\varphi)|\le\,\varepsilon\,\|\vu_n\|_{L^2(\T^3)}\|(\mathbb{I}-\mathcal{P}_m)\bm\varphi\|_{H^2(\T^3)}$$
To finish this, we procced with each term step by step and estimate the consistency errors. Let $\psi\in C^\infty(\mathbb{T}^3)$ be divergence free test function. then

\noindent
\textbf{Convective term:}
\begin{align*}
	\int_{\T^3}\bm\varphi\cdot\mathcal{P}_n(\vu_n\cdot\nabla\vu_n)\dx=\int_{\T^3}\dv(\vu_n\otimes\vu_n)\cdot\bm\varphi\dx + \mathcal{R}_1(n,\bm\varphi)
	\end{align*}
where $\mathcal{R}_1$ estimated as follows
\begin{align*}
	\mathcal{R}_1(n,\bm\varphi)&=\int_{\T^3}\bm\varphi\cdot\mathcal{P}_n(\vu_n\cdot\nabla\vu_n)\dx-\int_{\T^3}\dv(\vu_n\otimes\vu_n)\cdot\bm\varphi\dx\notag\\
	&=\int_{\T^3}\bm\varphi\cdot\dv\big((\mathbb{I}-\mathcal{P}_n)(\vu_n\otimes\vu_n)\big)\dx\notag\\
	&=-\int_{\T^3}\nabla\bm\varphi:(\mathbb{I}-\mathcal{P}_n)(\vu_n\otimes\vu_n)\dx\notag\\
	&=-\int_{\T^3}\nabla(\mathbb{I}-\mathcal{P}_n)\bm\varphi:(\vu_n\otimes\vu_n)\dx	
	\end{align*}
It implies that
\begin{align*}
	|\mathcal{R}_1(n,\bm\varphi)|&\le\,\|\vu_n\|_{L^\infty(\T^3)}\|\vu_n\|_{L^2(\T^3)}\|\nabla(\mathbb{I}-\mathcal{P}_n)\bm\varphi\|_{L^2(\T^3)}\notag\\
	&\le\,C\,n^{1/2}\|\vu_n\|_{L^2(\mathbb{T}^3)}^2\|\nabla(\mathbb{I}-\mathcal{P}_n)\bm\varphi\|_{L^2(\T^3)}\notag\\
	&\le\,C\,\|\vu_n\|_{L^2(\T^3)}^2\|(\mathbb{I}-\mathcal{P}_n)\bm\varphi\|_{H^{3/2}(\T^3)}
\end{align*}
where $\|(\mathbb{I}-\mathcal{P}_n)\bm\varphi\|_{H^{3/2}(\T^3)}\,\to\,0$ as $n\to\,\infty$.

\noindent
\textbf{Diffusion term:}
\begin{align*}
	\mathcal{N}(n,\bm\varphi):=\varepsilon\int_{\T^3}\dv(\mathcal{Q}_n\nabla\vu_n)\bm\varphi\dx=\varepsilon\int_{\T^3}\dv(\mathbb{I}-\mathcal{P}_m)\nabla\bm\varphi\cdot\vu_n\dx=\varepsilon\int_{\T^3}(\mathbb{I}-\mathcal{P}_m)\Delta\bm\varphi\cdot\vu_n \dx
\end{align*}
It implies that
\begin{align*}
	|\mathcal{N}(n,\bm\varphi)|\le\,\varepsilon\,\|\vu_n\|_{L^2(\T^3)}\|(\mathbb{I}-\mathcal{P}_m)\bm\varphi\|_{H^2(\T^3)}
\end{align*}

\noindent
\textbf{Stochastic term:}
\begin{align*}
	\int_0^t\int_{\T^3}\bm\varphi\mathcal{P}_n\mathcal{\sigma}(\vu_n)\dx\,\D W (s)=\int_0^t\int_{\T^3}\bm\varphi\,\mathcal{\sigma}(\vu_n)\,\dx\,\D W (s)
\end{align*}

\section{Convergence of spectral method (semi-discrete scheme)}\label{Section 5}
In this section, we develop the tools required for the proof of convergence of the scheme. The scheme have smooth solutions and consequent passage to the limit gives the existence of a {\em dissipative measure-valued martingale} solution to the original equation. Nevertheless, the limit argument is quite technical and has to be done in several steps. It is based on the compactness method: the uniform energy estimates yield tightness of sequence of approximate solutions and thus, on another probability space, this sequence converges almost surely due to the Jakubowski-Skorokhod representation theorem. Let us now prepare the setup for our compactness method. To establish the tightness of the laws generated by the approximations, let us define the path space $\mathcal{K}$ to be the product of the following spaces:
\begin{align*}
	\mathcal{K}_\vu&= C_w([0,T];L_{\text{div}}^2(\mt)),&\mathcal{K}_W&=C([0,T];\mathfrak{U}_0),\\
	\mathcal{K}_{\mathcal{C}} &= \big(L^{\infty}(0,T; \mathcal{M}_b(\T^3)), w^* \big),&\mathcal{K}_{\mathcal{E}} &= \big(L^{\infty}(0,T; \mathcal{M}_b(\T^3)), w^* \big),\\
	\mathcal{K}_{\mathcal{D}}&= \big(L^{\infty}(0,T; \mathcal{M}_b(\T^3)), w^* \big) &\mathcal{K}_{\mathcal{V}} &= \big(L^{\infty}((0,T)\times \T^3; \mathcal{P}(\R^3)), w^* \big),\\
	\mathcal{K}_{\mathcal{G}}&=\big(L^{\infty}(0,T; \mathcal{M}_b(\T^3)), w^* \big), & \mathcal{K}_\mathcal{N}&=C([0,T]; \R),\\	
\end{align*}
Let us denote by $\lambda_{\vu_n}$, and $\lambda_{W_n}$ respectively, the law of $\vu_n$ and $W_n$ on the corresponding path space. Moreover, let $\lambda_{\mathcal{N}_n}$ denote the law of martingales $\mathcal{N}_n:=\sum_{k\geq1}\int_0^t\int_{\T^3}{ \vu_{n} \cdot {\mathcal{P}_n\mathcal{\sigma}_k}( \textbf{u}_n) } \,\D W$ on the corresponding path spaces. Furthermore, let $\lambda_{\mathcal{C}_n}$, $\lambda_{\mathcal{D}_n}$, $\lambda_{\mathcal{E}_n}$, and $\lambda_{{\mathcal{V}}_n}$ denote the law of 
\begin{align*}
	& \mathcal{C}_n:= {\textbf{u}_n\otimes \textbf{u}_n}, \quad \mathcal{D}_{n}:= \frac{1}{2}|\mathcal{\sigma}({\bf u}_{n})|^2,\quad \mathcal{E}_n :={\frac{1}{2}|{\bf u}_n|^2}, \quad \quad {\mathcal{V}}_n := \delta_{ \textbf{u}_n}, \quad \mathcal{G}_n:=\frac{1}{2}|\mathcal{Q}_H{\mathcal\sigma(u_n)}|^2,
\end{align*}
respectively, on the corresponding path spaces. Finally, let $\lambda^n$ denotes joint law of all the variables on $\mathcal{K}$. To proceed further, it is necessary to establish tightness of $\{\lambda^n;\,n\in\mathbb{N}\}$. To this end, we observe that tightness of $\lambda_{W}$ is immediate. So we show tightness of other variables. We can easily prove the following uniform estimate which helps to conclude that laws given by approximate solutions is tight.
\begin{Lemma}[\textbf{Compactness in time}]
	Let $\vu_n$ be the solution of the semi-discrete scheme \eqref{approximate equations}. Then there exists $0\,\textless\,\alpha\,\textless\,\frac{1}{2}$, $C\textgreater\,0$ such that, for $\gamma\,\textgreater\,\frac{5}{2}$, $r\textgreater\,2$
	\begin{align}\label{compactness in time}
		\mathbb{E}\big[\|\vu_n\|_{C^{\alpha}([0,T],W^{-\gamma,2}(\mathbb{T}^3))}\big]\,\le\,C,
	\end{align}
and
\begin{align}\label{compactness in time 1}
	\mathbb{E}\Bigg[\Big\|\int_0^{\cdot}\int_{\T^3}\vu_n\mathcal{P}_n(\mathcal{\sigma}(\vu_n))\dx\D W\Big\|_{C^{\alpha}([0,T];R)}^r\Bigg]\,\le\,C.
\end{align}
\begin{proof}
	For a proof, we refer to \cite{K2}. 
\end{proof}
\end{Lemma}
\subsection{Stochastic compactness}
	\textbf{Tightness of law:}  To proceed, it is necessary to establish tightness of $\{\lambda^n; n\in\mathbb{N}\}$. Infact, we have all in hand to conclude our compactness argument by showing tightness of a certain collection of laws. 
\begin{Lemma}
	$\{\lambda^n,n\in\mathbb{N}\}$ is tight on $\mathcal{K}$.
\end{Lemma}
\begin{proof}
	Compact imbeddings give tightness of laws. For a proof, we refer to \cite{K2}.
	\end{proof}

\noindent
Since, the path space $\mathcal{K}$ is not a Polish space, our compactness argument is based on the modified version of ~Jakubowski-Skorokhod representation theorem instead of the classical Skorokhod representation theorem. To be more precise, passsing to a weakly convergent subsequence $\lambda^n$ and denoting by $\lambda$ the limit law, we infer the following result.
\begin{Proposition}\label{prop:skorokhod1}
	There exists a subsequence $\lambda^n$ (not relabelled), a probability space $(\widetilde\Omega,\widetilde{\mathcal{F}},\widetilde\prst)$ with $\mathcal{K}$-valued Borel measurable random variables $(\widetilde{\textbf{u}}_n, \widetilde W_n\, \widetilde {\mathcal{C}}_n, \widetilde {\mathcal{D}}_n, \widetilde {\mathcal{E}}_n, \widetilde {\mathcal{N}}_n, \widetilde{\mathcal{G}}_n, \widetilde {\mathcal{V}}_{n})$, $n\in\mathbb{N}$, and\\  $(\widetilde{\textbf{u}},\widetilde W, \widetilde{\mathcal{C}}, \widetilde{\mathcal{D}}, \widetilde{\mathcal{E}}, \widetilde {\mathcal{N}},\widetilde{\mathcal{G}}, \widetilde{\mathcal{V}})$ such that 
	\begin{enumerate}
		\item [(1)]the law of $(\widetilde{\textbf{u}}_n, \widetilde W_n\, \widetilde {\mathcal{C}}_n, \widetilde {\mathcal{D}}_n, \widetilde {\mathcal{E}}_n, \widetilde {\mathcal{N}}_n, \widetilde{\mathcal{G}}_n, \widetilde {\mathcal{V}}_{n})$ is given by $\lambda^n$, $n\in\mathbb{N}$,
		\item [(2)]the law of $(\widetilde{\textbf{u}},\widetilde W, \widetilde{\mathcal{C}}, \widetilde{\mathcal{D}}, \widetilde{\mathcal{E}}, \widetilde {\mathcal{N}},\widetilde{\mathcal{G}}, \widetilde{\mathcal{V}})$, denoted by $\mu$, is a Radon measure,
		\item [(3)]$(\widetilde{\textbf{u}}_n, \widetilde W_n\, \widetilde {\mathcal{C}}_n, \widetilde {\mathcal{D}}_n, \widetilde {\mathcal{E}}_n, \widetilde {\mathcal{N}}_n, \widetilde{\mathcal{G}}_n, \widetilde {\mathcal{V}}_{n})$ converges $\,\widetilde{\prst}$-almost surely to \\$(\widetilde{\textbf{u}},\widetilde W, \widetilde{\mathcal{C}}, \widetilde{\mathcal{D}}, \widetilde{\mathcal{E}}, \widetilde {\mathcal{N}},\widetilde{\mathcal{G}}, \widetilde{\mathcal{V}})$ in the topology of $\mathcal{K}$, i.e.,
		\begin{align*}
			&\widetilde{\textbf{u}}_n\rightarrow \bar{\textbf u} \,\, \text{in}\, \,C_w([0,T]; L_{\dv}^2(\T^3)),\quad&\widetilde W_n \rightarrow \widetilde W \,\, \text{in}\, \,C([0,T]; \mathcal{U}_0)),\\
			&\widetilde {\mathcal{C}}_n \rightarrow \widetilde {\mathcal{C}} \,\, \text{weak-$*$ in}\, \, L_{w*}^{\infty}(0,T; \mathcal{M}_b(\T^3)), \qquad & \widetilde {\mathcal{D}}_n \rightarrow \widetilde {\mathcal{D}} \,\, \text{weak-$*$ in}\, \, L_{w*}^{\infty}(0,T; \mathcal{M}_b(\T^3)), \\
			&\widetilde{\mathcal{N}}_n \rightarrow \widetilde N \,\, \text{in}\, \, C([0,T]; \R),\qquad
			& \widetilde {\mathcal{E}}_n \rightarrow \widetilde {\mathcal{E}} \,\, \text{weak-$*$ in}\, \, L_{w*}^{\infty}(0,T; \mathcal{M}_b(\T^3)), \\
			& \widetilde {\mathcal{V}}_n \rightarrow \widetilde {\mathcal{V}} \,\, \text{weak-$*$ in}\, \, L_{w*}^{\infty}((0,T)\times \T^3; \mathcal{P}(\R^3)),\qquad&\widetilde{\mathcal{G}}_n\rightarrow\widetilde{\mathcal{G}}\,\,\text{weak-$*$ in}\,\,L_{w*}^{\infty}(0,T; \mathcal{M}_b(\T^3)),
		\end{align*}
		\item [(4)] For any $n\in\mathbb{N}$, $\widetilde W_{n}=\widetilde W$.
		
	\end{enumerate}
\end{Proposition}
\begin{proof}
Proof of the items $(1)$, $(2)$, and $(3)$ directly follow from Jakubowski-Skorokhod representation theorem. For the proof of the item $(4)$, we refer to Theorem \ref{newth}, and \cite{Jakubowski}.
\end{proof}

\noindent
\textbf{Martingale solution:} In the following result, we will show that $\widetilde{\vu}_n$ is also  a solution of the approximate scheme \eqref{approximate equations} in another probability space.
\begin{Proposition}\label{prop:martsol}
	For every $n\in\mathbb{N}$, $\big((\widetilde{\Omega},\widetilde{\mathcal{F}},(\widetilde{\mathcal{F}}_t)_{t\ge\,0},\widetilde{\prst}),\widetilde{\vu}_n,\widetilde{W}\big)$ is a finite energy martingale solution to \eqref{approximate equations} with the initial data $\mathcal{T}_n(\vu_0)$. 
\end{Proposition}

\begin{proof}
	Proof of the above proposition directly follows form the Theorem 2.9.1 of the monograph by Breit et. al. \cite{BrFeHobook}.
\end{proof}

\noindent We note that the above proposition implies that the new random variables satisfy the follwoing equations and the energy inequality on the new probability space, 
\begin{itemize}
\item for all  $\bm\varphi\in C_{\text{div}}^\infty(\mathbb{T}^3)$ we have
	\begin{equation}\label{eq:energyt}
		\begin{aligned}
			\langle \mathbf{\widetilde{u}}_{n}(t), \bm\varphi\rangle &= \langle \mathbf{\widetilde{u}}_{n} (0), \bm\varphi\rangle - \int_0^{t}\langle \mathbf{\widetilde{u}}_{n}\otimes\mathbf{\widetilde{u}}_{n}, \nabla_x  \bm\varphi\rangle\,\mathrm{d}s
			+ {\varepsilon}\,\int_0^{t}\langle(\mathbb{I}-\mathcal{P}_m)\nabla_x \mathbf{\widetilde{u}}_{n}\,, \nabla_x  \bm\varphi\rangle\mathrm{d}s 
			\\&
			+\int_0^{t}\langle{\mathcal{\sigma}}(\mathbf{\widetilde{u}}_{n}), \bm\varphi\rangle\,\mathrm{d}W + \widetilde{\mathcal{R}}(n,\bm\varphi) + \widetilde{\mathcal{N}}(n,\bm\varphi)
		\end{aligned}
	\end{equation}
	$\widetilde{\mathbb{P}}$-a.s. for all $t\in[0,T]$, 
	\noindent
	where $\mathcal{\widetilde R}(h,\varphi),$ and $\mathcal{\widetilde N}(h,\bm{\varphi})$ are defined similarly as in \eqref{consistency formulation}, in the new probability space.
	\item the energy inequality\index{energy inequality}, $\widetilde{\p}$-a.s., for all $s\,\textless\,t\in[0,T]$ 
	\begin{align}\label{energy inequalty 1}
			&\frac{1}{2}\|\widetilde{\vu}_n(t)\|_{L^2(\T^3)}^2 + \ \varepsilon\int_s^t\|\mathcal{Q}_n(\nabla\widetilde{\vu}_n(s))\|_{L^2(\mathbb{T}^3)}^2\D s\,\notag\\
			&\le\,\frac{1}{2}\|\widetilde{\vu}_n(s)\|_{L^2(\mathbb{T}^3)}^2 + \int_s^t\int_{\mathbb{T}^3}\widetilde{\vu}_n(s)\mathcal{P}_n(\mathcal{\sigma}(\widetilde{\vu}_n(s)))\D\widetilde{ W}(s) + \int_s^t\|\mathcal{P}_n(\widetilde{\vu}_n(s))\|_{L^2(\mathbb{T}^3)}^2\D s
	\end{align}
\end{itemize}

\noindent
\textbf{Filtration:}  Note that, since $(\vu_n,\mathcal{N}_n)$ are random variables with values in $C([0,T];L^2(\mathbb{T}^3))\times C([0,T];\R)$. By \cite[Lemma A.3]{GA} and \cite[Corollary A.2]{Ondre},  $(\widetilde\vu_n,\widetilde{\mathcal{N}}_n)$ are also random variables with values in $C([0,T],L^2(\mathbb{T}^3))\times C([0,T];\R)$. Let $(\widetilde{\mathcal{F}}_t^n)$ be the $\widetilde{\prst}$-augmented canonical filtration of the process $(\widetilde\vu_n,\widetilde{W}, \widetilde{\mathcal {N}}_n)$, that is 
\begin{equation*}
\begin{split}
	\widetilde{\mathcal{F}}_t^n&=\mathcal{\sigma}\big(\mathcal{\sigma}\big(\mathbf{r}_t\widetilde\vu_n,\,\,\mathbf{r}_t \widetilde{W},  \mathbf{r}_t\widetilde{\mathcal{N}}_n\big)\cup\big\{N\in\widetilde{\mathcal{F}};\;\widetilde{\prst}(N)=0\big\}\big),\quad t\in[0,T],
\end{split}
\end{equation*}
where we denote by $\mathbf{r}_t$ the operator of restriction to the interval $[0,t]$ acting on various path spaces.
Let us remark that by assuming that the initial filtration $({\mathcal{F}}_t)_{t\ge\,0}$ is the one generated by ${W}$, by \cite[Lemma A.6]{GA}, one can consider $(\widetilde{\mathcal{F}}_t^h)=(\widetilde{\mathcal{F}_t})$ is the filtration generated by $\widetilde{W}$.

\noindent
\textbf{Almost surely limit:}
 The lack of strong convergence of  sequence $\widetilde{\vu}_n$ does not allow us to identify the limit of  the terms where the dependence on $\widetilde{\vu}_n$ is nonlinear, namely, the convective term in momentum equation  and  nonlinear terms in energy inequality. Next we want to pass the limit $n\to\infty$ in \eqref{eq:energyt} and \eqref{energy inequalty 1}. To complete this, first we recall that a-priori bounds \eqref{uniform estimate} which are remain hold for the new random variables. Young measure capture the weak limit \cite[section 2.8]{BrFeHobook}. Thus, by implementation of \cite[Theorem 2.8]{BrFeHobook}, we conclude that $\widetilde \p$-a.s., 
\begin{align}\label{weak limit}
	\widetilde {\textbf u}_n \rightharpoonup \langle {\mathcal{\widetilde V}^{\omega}_{t,x}}; \widetilde {\textbf u} \rangle, \,\,\text{weakly in}\,\, L^2((0,T);L_{\text{div}}^2(\T^3)).
\end{align}
 We first introduce the following concentration defect measures to pass limit in the nonlinear terms present in the equations.
\begin{align*}
	\widetilde \lambda_{\mathcal{C}}= \widetilde {\mathcal{C}} -\left\langle \mathcal{\widetilde V}^{\omega}_{(\cdot, \cdot)}; {\widetilde {\bf u}\otimes \widetilde {\bf u}} \right\rangle dxdt,\,\,\widetilde \lambda_{\mathcal{E}} = \widetilde {\mathcal{E}}- \left\langle \mathcal{\widetilde V}^{\omega}_{(\cdot, \cdot)}; \frac{1}{2} {|\widetilde {\bf u}|^2} \right \rangle dx,\,\,\widetilde \lambda_{\mathcal{D}}= \widetilde {\mathcal{D}} -\left\langle \mathcal{\widetilde V}^{\omega}_{(\cdot, \cdot)};\frac{1}{2} |{\mathcal{\sigma}} (\widetilde {\bf u}) |^2 \right\rangle dxdt. 
\end{align*}
\begin{align*}
	\widetilde{\lambda}_{\mathcal{G}}=\widetilde{\mathcal{G}}-\frac{1}{2}\bigg(\mathcal{Q}_H\left\langle \mathcal{\widetilde V}^{\omega}_{(\cdot, \cdot)}; |{\mathcal{\sigma}} (\widetilde {\bf u}) | \right\rangle\bigg)^2
\end{align*}
Make use of these concentration defect measures, we conclude that $ \mathbb{\widetilde P}$-a.s.
\begin{align*}
	&\widetilde {\mathcal{C}}_n \rightharpoonup \left\langle \mathcal{\widetilde V}^{\omega}_{(\cdot, \cdot)}; {\widetilde {\bf u}\otimes \widetilde {\bf u}}\right\rangle dxdt + \widetilde \lambda_{\mathcal{C}}, \,\, \text{ weak-$*$ in}\, \, L_{w*}^{\infty}(0,T; \mathcal{M}_b(\T^3)), \\
	&\widetilde {\mathcal{D}}_n \rightharpoonup \left\langle \mathcal{\widetilde V}^{\omega}_{(\cdot, \cdot)};\frac{1}{2}|{ \mathcal{\sigma}} (\widetilde {\bf u}) |^2\right\rangle dxdt + \widetilde \lambda_{\mathcal{D}}, \,\, \text{weak-$*$ in}\, \, L_{w*}^{\infty}(0,T; \mathcal{M}_b(\T^3)), \\
	&\widetilde {\mathcal{E}}_n \rightharpoonup \left\langle \mathcal{\widetilde V}^{\omega}_{(\cdot, \cdot)}; \frac{1}{2}|\widetilde {\bf u}|^2\right\rangle dxdt + \widetilde \lambda_{\mathcal{E}}, \,\, \text{weak-$*$ in}\, \, L_{w*}^{\infty}(0,T; \mathcal{M}_b(\T^3)),\\
	&\widetilde {\mathcal{G}}_n \rightharpoonup\frac{1}{2}\bigg(\mathcal{Q}_H\left\langle \mathcal{\widetilde V}^{\omega}_{(\cdot, \cdot)};  |{\mathcal{\sigma}}(\widetilde {\bf u})|  \right\rangle\bigg)^2+ \widetilde{\lambda}_{\mathcal{G}} \,\, \text{weak-$*$ in}\, \, L_{w*}^{\infty}(0,T; \mathcal{M}^+_b(\T^3)).
\end{align*}

\noindent
\textbf{Pass to limit in approximation of the momentum equation:}
By making use of \cite[Theorem 2.8]{BrFeHobook}, we have $\widetilde\p$-a.s.,
\begin{align}\label{weak limit 2}\mathcal{\sigma}(\widetilde{\vu}_n)\rightharpoonup\left\langle\widetilde{\mathcal{V}}_{t,x}^\omega;\mathcal{\sigma}(\widetilde{\vu})\right\rangle\,\,\text{weakly \,in}\,\,L^2([0,T];L^2(\mathbb{T}^3)).
	\end{align}
Note that the It\^o integral
$$ I_t:\varphi \rightarrow \int_0^t \varphi(s) \D \widetilde {W}(s) $$
is a linear and continuous (hence weakly continuous) map from $L^2(\Omega \times [0,T]; L^2(\T^3))$ to $L^2(\Omega; L^2(\T^3))$.
Therefore, we can make use of weak continuity of It\^{o} integral, and item $(4)$ of Proposition~\ref{prop:skorokhod1}, to conclude $I_t(\mathcal{\sigma} (\widetilde {\bf u}_n) )$ converges weakly to $I_t\big(\langle {\mathcal{\widetilde V}}^{\omega}_{t,x} ; \mathcal{\sigma} ( \widetilde{\bf u}) \rangle\big)$ in $L^2(\Omega;L^2(\mathbb{T}^3))$. Make use of above information and energy bounds we can conclude that
\begin{align}\label{above}
	\int_{\widetilde\Omega}\big\langle\langle\widetilde{\mathcal{V}}_{t,x}^\omega;\widetilde{\vu}\rangle,\bm{\varphi}\big\rangle \alpha({\omega})\D\widetilde\p({{\omega}})&=\int_{\widetilde\Omega}\Big[\big\langle\langle \widetilde{\mathcal{V}}_{0,x}^\omega;\widetilde{\vu}\rangle,\bm{\varphi}\big\rangle + \int_{0}^t\big\langle\langle\widetilde{\mathcal{V}}_{s,x};\widetilde{\vu}\otimes\widetilde{\vu}\rangle,\nabla\bm{\varphi}\big\rangle \D s+\int_{0}^t\int_{\T^3}\nabla\bm{\varphi}\D\widetilde{\lambda}_{\mathcal{C}}\notag\\&\qquad + \int_{0}^t\big\langle\langle\widetilde{\mathcal{V}}_{s,x}^\omega;\mathcal{\sigma}(\widetilde{\vu})\rangle,\bm{\varphi}\big\rangle\D \widetilde{W}(s)\Big]\alpha({\omega})\D\widetilde\p({{\omega}})
\end{align}
holds for all $t\in [0,T)$, for all $\alpha\in L^2(\widetilde{\Omega})$ and for all $\bm{\varphi} \in C^{\infty}(\T^3;\mathbb{R}^3)$. Since $C^\infty(\mathbb{T}^3)$ is separable space with sup norm, above equation \eqref{above} implies that
\begin{align}
	\big\langle\langle\widetilde{\mathcal{V}}_{t,x}^\omega;\widetilde{\vu}\rangle,\bm{\varphi}\big\rangle&=\big\langle\langle \widetilde{\mathcal{V}}_{0,x}^\omega;\widetilde{\vu}\rangle,\bm{\varphi}\big\rangle + \int_{0}^t\big\langle\langle\widetilde{\mathcal{V}}_{s,x};\widetilde{\vu}\otimes\widetilde{\vu}\rangle,\nabla\bm{\varphi}\big\rangle \D s+\int_{0}^t\int_{\T^3}\nabla\bm{\varphi}\D\widetilde{\lambda}_{\mathcal{C}}\notag\\&\qquad + \int_{0}^t\big\langle\langle\widetilde{\mathcal{V}}_{s,x}^\omega;\mathcal{\sigma}(\widetilde{\vu})\rangle,\bm{\varphi}\big\rangle\D \widetilde{W}(s)
\end{align} 
for all $t\in[0,T]$, $\widetilde{\p}$-a.s., for all $\bm{\varphi}\in C_{\dv}^{\infty}(\T^3)$.

\noindent
 { \textbf{$\mathbf{\tilde{\lambda}_{\mathcal{G}}}$ is nonnegative measure:}} It is clear from \eqref{weak limit 2} that $\widetilde{\p}$-a.s.,
\begin{align*}
	\mathcal{Q}_H(\mathcal{\sigma}(\widetilde{\vu}_n))\rightharpoonup \mathcal{Q}_H\big(\langle\widetilde{\mathcal{V}}_{t,x}^\omega;\mathcal{\sigma}(\widetilde{\vu})\rangle\big)\,\mbox{weakly in}\,L^2([0,T];(L^2(\T^3))^\perp).
\end{align*}
Let $\psi\in C([0,T]\times\T^3)$ with $\psi\,\ge 0$. Making use of weakly lower semi-continuity of norm, then we have $\widetilde\p$-a.s.,
\begin{align*}
	\Big\langle\psi\,\,,\,\,  |\mathcal{Q}_H\big(\langle\widetilde{\mathcal{V}}_{t,x}^\omega;\mathcal{\sigma}(\widetilde{\vu})\rangle\big)|^2\Big\rangle\le\,\liminf_{n\to\infty}\Big\langle\psi\,\,,\,\,  |\mathcal{Q}_H\big(\mathcal{\sigma}(\widetilde\vu_n)\big)|^2\Big\rangle
\end{align*}
It shows that $\tilde{\lambda}_{\mathcal{G}}$ is non-negative measure.

\noindent
\textbf{Energy inequality and concentration defect:} 
\noindent In this subsection, we show that appropriate form of energy inequality also holds for dissipative measure-valued martingale solutions in the following  three steps.

\noindent
\textbf{Step-1. Martingale term.}
\begin{Proposition}
	 $\widetilde{\p}$-a.s., $\widetilde {\mathcal{N}}_{n} \rightarrow \widetilde {\mathcal{N}}$ in $C([0,T];\R)$, and for all $t\in[0,T]$,  $\widetilde {\mathcal{N}}(t)$ is a real valued square-integrable martingale.
\end{Proposition}
\begin{proof}
	Note that, thanks to Proposition~\ref{prop:skorokhod1}, we have the information $\widetilde {\mathcal{N}}_{n} \rightarrow  \widetilde  {\mathcal{N}},$ $\widetilde{\p}$-a.s. in $C([0,T]; \R)$. To conclude that $\widetilde  {\mathcal{N}}(t)$ is a martingale, it is enough to show that
	$$\widetilde{\mathbb{E}}[\widetilde  {\mathcal{N}}(t)| \mathcal{\widetilde F}_s] = \widetilde{\mathcal{N}}(s),
	$$
	for all $t,s \in [0,T]$ with $s \le t$. To prove this, we have to show that, for $A\in\widetilde{\mathcal{F}}_s$
	$$ \widetilde{\mathbb{E}}\Big[ \mathcal{I}_{A} \big(\widetilde  {\mathcal{N}}(t)-\widetilde  {\mathcal{N}}(s)\big) \Big]=0,
	$$ 
	Now using the information that $\widetilde  {\mathcal{N}}_{n}(t)$ is a martingale, we know that
	$$\widetilde{\mathbb{E}}\Big[ \mathcal{I}_{A} \big(\widetilde  {\mathcal{N}}_{n}(t)-\widetilde  {\mathcal{N}}_{n}(s)\big) \Big]=0,
	$$
	for all $A\in\widetilde{\mathcal{F}}_s$. For each t, ${\widetilde{\mathcal{N}}_n}$ is uniformly bounded in $L^2(\widetilde{\Omega})$, with the help of Vitali's convergence theorem, we can pass to the limit in $n$ to conclude that $\widetilde  {\mathcal{N}}(t)$ is a martingale. In this manner, we insecure the structure of the martingale $\widetilde  {\mathcal{N}}(t)$, which is expected due of lack of sufficient regularity.
\end{proof}

\noindent
\textbf{Step 2. Control on concentration defect measures:}
\begin{Lemma}\label{rhoutight131}
	The concentration defect $0\le \mathcal{\widetilde H}(r):= \widetilde \lambda_{\mathcal{E}}(r)(\T^3)$ dominates defect measures $\widetilde \lambda_{\mathcal{D}}\,\, \&\,\,  \widetilde{\lambda}_{\mathcal{C}}$. More precisely, there exists a constant $C>0$ such that
	\begin{equation*} 
		\int_{0}^{r} \int_{\T^3} d|\widetilde \lambda_{\mathcal{C}}| + \int_{0}^{r} \int_{\T^3} d|\widetilde \lambda_{\mathcal{D}}|\leq C \int_{0}^{r} \mathcal{\widetilde H}(r)\,dt,
	\end{equation*}	
	$\widetilde{\p}$-a.s., for all $r \in (0,T)$.
\end{Lemma}
\begin{proof}
	With the help of \cite[Lemma 2.3]{MKS01}, it is clear that $\widetilde \lambda_{\mathcal{E}}$ dominates defect measures $\widetilde \lambda_{\mathcal{C}} $. To show the dominance of $\widetilde \lambda_{\mathcal{E}}$ over $\widetilde \lambda_{\mathcal{D}}$, observe that, by virtue of hypotheses (\ref{FG1}), (\ref{FG2}), the function
	\[
	[\vc{u}] \mapsto { |{\mathcal{\sigma}}({\bf u})|^2 } \ \mbox{is continuous},
	\]
	and as such dominated by the total energy
	\[
	 { |{ \mathcal{\sigma}}({\bf u}) |^2 }\leq c \left( 1 + {| {\bf u} |^2} \right) 
	\]
	Hence, a consequence of \cite[Lemma 2.3]{MKS01} completes the proof of the lemma. 
\end{proof}

\noindent
\textbf{Step 3. Energy inequality.} 
Use the fact that $\|\mathcal{P}_n(\mathcal{\sigma}(\widetilde{\vu}_n))\|_{L^2(\T^3)}^2\le\|\mathcal{P}_H(\mathcal{\sigma}(\widetilde{\vu}_n))\|_{L^2(\T^3)}^2$ then from \eqref{L0}, we have $\widetilde{\p}$-a.s, for all $s\,\textless\,t\in[0,T]$
\begin{align*}\label{L4}
	\frac{1}{2}\|\widetilde{\vu}_n(t)\|_{L^2(\T^3)}^2 &\le\,\frac{1}{2}\|\widetilde{\vu}_n(s)\|_{L^2(\mathbb{T}^3)}^2 + \int_s^t\int_{\mathbb{T}^3}\widetilde{\vu}_n(s)\mathcal{P}_n(\mathcal{\sigma}(\widetilde{\vu}_n(s)))\D W(s) + \int_s^t\|\mathcal{P}_n(\mathcal{\sigma}(\widetilde{\vu}_n))\|_{L^2(\mathbb{T}^3)}^2\D s
\end{align*}
Now we would like to pass limit $n\to\infty$ in energy inequality. Let $r,\,\delta\textgreater\,0$ (small enough), then from energy inequality, $\widetilde{\p}$-a.s., for all $s\,\textless\,\tau\in[0,T]$ 
\begin{align*}
	\frac{1}{2r}\int_\tau^{\tau+r}\int_{\T^3}|\widetilde{\vu}_n(t)|^2\dx\D t  \le\,&\frac{1}{2\delta}\int_{s-\delta}^s\int_{\T^3}|\widetilde{\vu}_n(a)|^2\dx\D a\\& +\frac{1}{r}\int_{\tau}^{\tau+r}\Big(\frac{1}{\delta}\int_{s-\delta}^s\int_a^b\int_{\mathbb{T}^3}\widetilde{\vu}_n(s)\mathcal{P}_n(\mathcal{\sigma}(\widetilde{\vu}_n(t)))\D W(t)\D a\Big)\D b \\&+\frac{1}{r}\int_{\tau}^{\tau+r}\Big(\frac{1}{\delta}\int_{s-\delta}^s\int_a^b\|\mathcal{\sigma}(\widetilde{\vu}_n(t))\|_{L^2(\mathbb{T}^3)}^2\D t\D a\Big)\D b\\&-\frac{1}{r}\int_\tau^{\tau+r}\Big(\frac{1}{\delta}\int_{s-\delta}^s\int_a^b\int_{\T^3}|\mathcal{Q}_H(\mathcal{\sigma}(\widetilde{\vu}_n(t)))|^2\D x\D t \D a\Big)\D b
\end{align*}
Now pass to limit $n\to\infty$ in above equation and make use of information of limits. Then $\widetilde\p$-a.s., for all $s\,\textless\,\tau\in[0,T]$
\begin{align*}
	\frac{1}{r}\int_\tau^{\tau+r}&\int_{\T^3}\left\langle\widetilde{\mathcal{V}}_{t,x}^\omega;\frac{1}{2}|\widetilde{\vu}|^2\right\rangle\dx\D t + \int_\tau^{\tau+r}\widetilde{\mathcal{H}}(t)\D t\,\le\,\frac{1}{\delta}\int_{s-\delta}^s\int_{\T^3}\left\langle\widetilde{\mathcal{V}}_{a,x}^\omega;\frac{1}{2}|\widetilde{\vu}|^2\right\rangle\dx \D a\\&+\frac{1}{r}\int_{\tau}^{\tau+r}\Big(\frac{1}{\delta}\int_{s-\delta}^s\big(\widetilde{\mathcal{N}}(t)-\widetilde{\mathcal{N}}(a)\big)\D a\Big)\D t\notag\\& +\frac{1}{r}\int_{\tau}^{\tau+r}\Big(\frac{1}{\delta}\int_{s-\delta}^s\int_a^t\int_{\T^3}\Big(\left\langle \mathcal{\widetilde V}^{\omega}_{(a,x)};\frac{1}{2}|{ \mathcal{\sigma}} (\widetilde {\bf u}) |^2\right\rangle + \D\widetilde \lambda_{\mathcal{D}}(x,a)\Big) \dx\D a\Big)\D t\\& -\frac{1}{r}\int_{\tau}^{\tau+r}\Big(\frac{1}{\delta}\int_{s-\delta}^s\int_a^t\int_{\T^3}\Big(\bigg(\mathcal{Q}_H\left\langle \mathcal{\widetilde V}^{\omega}_{(a,x)}; \frac{1}{2} |{\mathcal{\sigma}}(\widetilde {\bf u})|  \right\rangle\bigg)^2-\D \widetilde{\lambda}_{\mathcal{G}}(x,a)\Big)\dx\D a\Big)\D t
\end{align*}
Letting $r, \delta\to 0$ and use that $\widetilde{\lambda}_{\mathcal{G}}$ is non-negative measure, then we have $\widetilde\p$-a.s., for all $s\,\textless\,\tau\in[0,T]$
\begin{align} 
	\lim_{r\to 0}\frac{1}{r}\int_\tau^{\tau+r}&\Big[\int_{\T^3}\left\langle\widetilde{\mathcal{V}}_{t,x}^\omega;\frac{1}{2}|\widetilde{\vu}|^2\right\rangle\dx\D t +\widetilde{\mathcal{H}}(t)\Big]\D t\,\le\,\lim_{\delta\to 0}\frac{1}{\delta}\int_{s-\delta}^s\Big[\int_{\T^3}\left\langle\widetilde{\mathcal{V}}_{a,x}^\omega;\frac{1}{2}|\widetilde{\vu}|^2\right\rangle\dx +\widetilde{\mathcal{H}}(a)\Big]\D a\notag\\&+\widetilde{\mathcal{N}}(\tau)-\widetilde{\mathcal{N}}(s)+\int_s^\tau\int_{\T^3}\left\langle \mathcal{\widetilde V}^{\omega}_{(t,x)};\frac{1}{2}|{\mathcal{\sigma}} (\widetilde {\bf u}) |^2\right\rangle \dx\D t + \int_s^\tau\int_{\T^3}\D\widetilde \lambda_{\mathcal{D}}(x,s)\D t\notag\\&  -\int_s^\tau\int_{\T^3}\bigg(\mathcal{Q}_H\left\langle \mathcal{\widetilde V}^{\omega}_{(t,x)}; \frac{1}{2} |{\mathcal{\sigma}}(\widetilde {\bf u})|  \right\rangle\bigg)^2\dx\D t
\end{align}
\section{Proof of Theorem~\ref{dissipative solution}: Convergence to a dissipative solution}\label{Section 6}
\label{dissipative solution 1}

With the help of the Proposition~\ref{prop:skorokhod1}, and convergence results given by \eqref{weak limit}, we conclude that there exists a subsequence $\widetilde{\vu}_{n_k}$ such that $\widetilde{\p}$-a.s,
$$\widetilde{\vu}_{n_k}\to\langle {\widetilde{\mathcal{V}}^{\omega}_{t,x}}; \widetilde{\vu} \rangle\,\,\mbox{in}\,\,\,C_w([0,T],L_{\dv}^2(\mathbb{T}^3)).$$
For the pointwise converegnce of approximations, we can make use of \cite[Proposition 2.4]{K1}. Indeed, we obtain $\widetilde{\p}$-a.s., there exists a subsequece $\widetilde{\vu}_{n_k}$ such that
\begin{equation*}
	\begin{aligned}
		\frac 1N \sum_{k=1}^N \widetilde{\vu}_{n_k} &\to\langle {\widetilde{\mathcal{V}}^{\omega}_{t,x}}; \widetilde{\vu} \rangle, \ \mbox{as $N \rightarrow \infty$ a.e. in} \,\,(0,T)\times\T^3.
	\end{aligned}
\end{equation*}	
\section{Weak-strong uniqueness}\label{Section 7}
	\textbf{Weak It\^o formula:} Here we give some outlines of proof of weak strong uniqueness. In this section, we prove Theorem~\ref{Weak-Strong Uniqueness_01} through some auxiliary results.
\begin{Lemma}[Weak It\^o formula]\label{Ito}
	Let $\mathbf{\bf V}$ be a stochastic process on $\big(\Omega,\mathfrak{F}, (\mathfrak{F}_{t})_{t\geq0},\mathbb{P} \big)$ such that 
	
	$$ \mathbf{V}\in C_w([0,T];L^2_{\dv }(\mathbb{T}^3))\cap L^{\infty}((0,T);L^2_{\dv}(\mathbb{T}^3))\,\qquad\,\mathbb{P}-\text{a.s.} $$
	
	$$\mathbb{E}\big[\sup_{t\in[0,T]}\|\mathbf{V}\|_{L^2_{\dv}(\mathbb{T}^3)}^2\big]\,\textless\,\infty,$$
	\begin{equation}\label{eq1}
		\begin{aligned}
		\int_{\mathbb{T}^3}\mathbf{V}(t)\cdot\bm{\varphi}\,dx=&\int_{\mathbb{T}^3}\mathbf{V}(0)\cdot\bm{\varphi}\,dx + \int_{0}^t\int_{\mathbb{T}^3}\mathbf{V}_1:\nabla\bm{\varphi}dx ds+\int_0^t\int_{\mathbb{T}^3}\nabla\bm{\varphi}:\,d\lambda(x,s) \,ds\\& + \int_0^t\langle\bm{\varphi},\mathcal{\sigma}(\mathbf{V})\rangle dW(s)
		\end{aligned}
	\end{equation}

	for all $\bm{\varphi}\in C^{\infty}(\T^3)$, for all $t\in[0,T]$, $\mathbb{P}$-a.s. Here $\mathbf{V}_1,\,\,\lambda $ satisfies with
	
	$$\mathbf{V}_1\in L^2(\Omega;L^1(0,T;L^1(\mathbb{T}^3))),\qquad \lambda\in L^1(\Omega;L_{w*}^\infty(0,T;\mathcal{M}_b(\mathbb{T}^3))).$$
	 
	Let $\mathbf{U}$ be a stochastic process on $\big(\Omega,\mathfrak{F}, (\mathfrak{F}_{t})_{t\geq0},\mathbb{P} \big)$ satisfying
	$$\mathbf{U}\in C([0,T]; C^1(\mathbb{T}^3)),\,\mathbb{P}-\text{a.s.}\,\,\text{and}\,\,\,\,\mathbb{E}\big[\sup_{t\in[0,T]}\|\mathbf{U}\|_{L_{\dv}^2(\mathbb{T}^3)\cap {C}(\mathbb{T}^3)}^2\big]\textless\,\infty,$$
	\begin{align}\label{eq}\D\mathbf{U}={\mathbf{U}_1}{\rm d}t + \mathbf{U}_2 \,\D W
	 \end{align}
	Here $\mathbf{U}_1, \mathbf{U}_2$ are progressible measurable with
	$$\mathbf{U}_1\in L^2(\Omega;L^1((0,T);L^2_{\dv}(\mathbb{T}^3)))\,\qquad\,\mathbf{U}_2\in L^2(\Omega;L^2((0,T);L_2(\mathfrak{U};L_{\dv}^2(\mathbb{T}^3))))$$.
	$$\sum_{k}^\infty\int_0^T\|\mathcal{P}_H\mathbf{U}_2(e_k)\|^2_{L^2(\mathbb{T}^3)}\in L^1(\Omega).$$
	
	Then, for all $t\in[0,T],\,\,\mathbb{P}$-a.s
	\begin{align}\label{eq4}
		\int_{\mathbb{T}^3}\mathbf{V}(t)\cdot\mathbf{U}(t)\dx&=\int_{\mathbb{T}^3}\mathbf{V}(0)\cdot\mathbf{U}(0)\dx + \int_0^t\int_{\mathbb{T}^3} \mathbf{V}_1:\nabla \mathbf{U} \dx\,{\rm d}s+\int_0^t\int_{\mathbb{T}^3}\nabla\mathbf{U}:\,\D\lambda\,\D s\notag\\
		&\qquad+\int_0^t\int_{\mathbb{T}^3}\mathbf{U}\cdot \mathcal{\sigma}(V)\dx \D W\,+\int_0^t\int_{\mathbb{T}^3}\mathbf{U}_1\cdot\mathbf{V}\,\dx\D s + \int_0^t\int_{\mathbb{T}^3}\mathbf{V}\cdot\mathbf{U}_2\,\D W\,\dx\notag\\
		&\qquad+\int_0^t\int_{\mathbb{T}^3}\mathcal{P}_{H}(\mathcal{\sigma} (\mathbf V))\mathcal{P}_{H}(\mathcal{\sigma}(\mathbf U))\dx \D s\,
	\end{align}
\end{Lemma}
\begin{proof}
	Let $\bm{\varphi}\in L^2_{\text{div}}(\mathbb{T}^3)$, then $\bm{\varphi}_r = \bm{\varphi}*\bm{\rho}_r\,\in C_{\text{div}}^\infty(\mathbb{T}^3)$, we have $\mathbb{P}$-a.s., for all $t\in[0,T]$
	\begin{align*}
		\int_{\mathbb{T}^3}\mathbf{V}(t)\cdot\bm{\varphi}_r\,dx&=\int_{\mathbb{T}^3}\textbf{V}(0)\cdot\bm{\varphi}_r\,dx + \int_{0}^t\int_{\mathbb{T}^3}\mathbf{V}_1:\nabla\bm{\varphi}_rdx ds+\int_0^t\int_{\mathbb{T}^3}\nabla\bm{\varphi}_r:\,d\lambda(x,s) \,ds \\&\qquad+ \int_{\mathbb{T}^3}\bm{\varphi}_r\cdot\int_0^t\mathcal{\sigma}(\mathbf{V})\D W(s) \dx
	\end{align*}
After shiftting molification on other variable, then we have $\mathbb{P}$-a.s., for all $t\in[0,T]$
\begin{align*}
	\int_{\mathbb{T}^3}\mathbf{V}_r(t)\cdot\bm{\varphi}\,dx&=\int_{\mathbb{T}^3}\textbf{V}_r(0)\cdot\bm{\varphi}\,dx + \int_{0}^t\int_{\mathbb{T}^3}(\mathbf{V}_1)_r:\nabla\bm{\varphi}dx ds+\int_0^t\int_{\mathbb{T}^3}\nabla\bm{\varphi}:\,d\lambda_r(x,s) \,ds\\&\qquad +\int_{\mathbb{T}^3}\bm{\varphi}\cdot\int_0^t\mathcal{\sigma}(V)_r\D W(s)\, \dx,
\end{align*}
It implies that $\mathbb{P}$-a.s., for all $t\in[0,T]$
\begin{align}\label{eq2}
	\mathbf{V}_r(t)&=\textbf{V}_r(0) -\int_0^t\mathcal{P}_H\big(\Div(\mathbf{V}_1)_r \big){\rm d}s-\int_0^t\,\mathcal{P}_H\big(\Div \lambda_r\big)  ds +\int_0^t\mathcal{P}_H(\mathcal{\sigma}(V)_r)\D W(s)\,
\end{align}
Let $(e_i)_{i\ge\,1}$ is countable orthonormal basis of $L^2(\T^3)$, and from equations \eqref{eq}-\eqref{eq2}, we have $\p$-a.s. , for all $t\in[0,T]$
\begin{align*}
	\langle \mathbf{U}(t),e_i\rangle=\langle \mathbf{U}(0),e_i\rangle + \int_0^T\langle\mathbf{U}_1(s),e_i\rangle \D s +\int_0^t\langle\mathcal{P}_H(\mathcal{\sigma}(\mathbf{U})),e_i\rangle\D W(s)
\end{align*}
and 
\begin{align*}
	\langle\mathbf{V}_r(t),e_i\rangle&=\langle\textbf{V}_r(0),e_i\rangle -\int_0^t\langle\mathcal{P}_H\big(\Div(\mathbf{V}_1)_r \big),e_i\rangle{\rm d}s-\int_0^t\,\langle\mathcal{P}_H\big(\Div \lambda_r\big),e_i\rangle  ds \notag\\&\qquad+\int_0^t\langle\mathcal{P}_H(\mathcal{\sigma}(V))_r,e_i\rangle\D W(s)\,
\end{align*}
Now, we apply It\^o product rule to $t\mapsto\langle \mathbf{U}(t),e_i\rangle\cdot\langle \mathbf{V}_r(t),e_i\rangle$, we have $\p$-a.s, for all $t\in[0,T]$
\begin{align*}
	\langle \mathbf{U}&(t),e_i\rangle\cdot\langle \mathbf{V}_r(t),e_i\rangle=\langle \mathbf{U}(0),e_i\rangle\cdot\langle \mathbf{V}_r(0),e_i\rangle + \int_0^t \langle\mathcal{P}_H\big(\Div(\mathbf{V}_1)_r \big),e_i\rangle\cdot\langle\mathbf{U}(s),e_i\rangle \D s\notag\\& + \int_0^t\langle\mathcal{P}_H\big(\Div \lambda_r\big),e_i\rangle\cdot\langle\mathbf{U}(s),e_i\rangle \D s + \int_0^t\langle\mathcal{P}_H(\mathcal{\sigma}(V))_r,e_i\rangle\cdot \langle\mathbf{U}(s),e_i\rangle\D W(s)\notag\\
	&+\int_0^t\langle\mathbf{V}_r(t),e_i\rangle\cdot\langle\mathbf{U}_1(s),e_i\rangle \D s + \int_0^t\langle\mathbf{V}_r(t),e_i\rangle\cdot\langle\mathcal{P}_H(\mathcal{\sigma}(\mathbf{U})),e_i\rangle\D W (s)\notag\\&+\int_{0}^t\langle\mathcal{P}_H(\mathcal{\sigma}(\mathbf{U})),e_i\rangle\cdot\langle\mathcal{P}_H(\mathcal{\sigma}(\mathbf{V}))_r,e_i\rangle\D s
\end{align*}
Now  we use the fact that $\int_{\T^3}\mathbf{U}\cdot\mathbf{V}\dx=\sum_{i\ge\,1}\langle \mathbf{U},e_i\rangle\cdot\langle\mathbf{V}, e_i\rangle$ then we have, $\p$-a.s., for all $t\in[0,T]$
\begin{align*}
	\int_{\T^3}\mathbf{U}&(t)\cdot\mathbf{V}_r(t)\dx= \int_{\T^3}\mathbf{U}(0)\cdot \mathbf{V}_r(0)\dx + \int_0^t \int_{\T^3}\mathcal{P}_H\big(\Div(\mathbf{V}_1) \big)_r\cdot\mathbf{U}(s)\rangle\dx \D s\notag\\& + \int_0^t\int_{\T^3}\mathcal{P}_H\big(\Div \lambda_r\big)\cdot\mathbf{U}(s)\dx \D s + \int_0^t\int_{\T^3}\mathcal{P}_H(\mathcal{\sigma}(V))_r\cdot \mathbf{U}(s)\dx\D W(s)\notag\\
	&+\int_0^t\int_{\T^3}\mathbf{V}_r(t)\cdot\mathbf{U}_1(s)\dx \D s + \int_0^t\int_{\T^3}\mathbf{V}_r(t)\cdot\mathcal{P}_H(\mathcal{\sigma}(\mathbf{U}))\dx\D W(s)\notag\\&+\int_{0}^t\int_{\T^3}\mathcal{P}_H(\mathcal{\sigma}(\mathbf{U}))\cdot\mathcal{P}_H(\mathcal{\sigma}(\mathbf{V}))_r\dx\D s.
\end{align*}
Now we are able to perform the limit $r\to 0$ in the above relation by using the hypotheses of Lemma \ref{Ito}, completing the proof.
\end{proof}
\subsection{Relative energy inequality (the Euler system)}
\label{MEI}
	\textbf{Relative energy functional:} We proceed further and introduce the \textit{relative energy (entropy)} functional. The commonly used form of the \textit{relative energy} functional in the context of measure-valued solutions to the incompressible Euler system reads, $\p$-a.s, for all $t\in[0,T]$
\begin{equation}
	\notag
	\begin{aligned} 
	\mathfrak{E}_{\mathrm{mv}} \left(\textbf{u} \ \Big|{\bf U} \right)(t)
		:=
		&\liminf_{r\to0}\frac{1}{r}\int_{t}^{t+r}\bigg[\int_{\T^3}{\left\langle {\mathcal{V}^{\omega}_{s,x}}; \frac{1}{2} { |\textbf{u}|^2}\right\rangle }\dx+ \mathcal{H}(s)\bigg]\D s
		- \int_{\T^3}{ \big \langle {\mathcal{V}^{\omega}_{t,x}}; \textbf{u} \big \rangle \,\cdot {\bf U(t)}}\dx\\&\qquad +\frac{1}{2}  \int_{\T^3}{  |{\bf U(t)}|^2 } \dx.
	\end{aligned}
\end{equation}
We note that by Lebesgue differentiation theorem and energy inequality \eqref{energy_001}, we have $\p$-a.s, almost every $t\in[0,T]$
\begin{equation}
	\notag
	\begin{aligned} 
		&\mathfrak{E}_{\mathrm{mv}} \left(\textbf{u} \ \Big|{\bf U} \right)(t)
		:=
	\int_{\T^3}{\left\langle {\mathcal{V}^{\omega}_{t,x}}; \frac{1}{2} { |\textbf{u}|^2}\right\rangle }\dx
		- \int_{\T^3}{ \big \langle {\mathcal{V}^{\omega}_{t,x}}; \textbf{u} \big \rangle \,\cdot {\bf U(t)}}\dx +\frac{1}{2}  \int_{\T^3}{  |{\bf U(t)}|^2 } \dx+ \mathcal{H}(t).
	\end{aligned}
\end{equation}
\textbf{Relative energy inequality:}
\begin{Proposition}[Relative Energy] 
	\label{relen}
	Let $\big[ \big(\Omega,\mathfrak{F}, (\mathfrak{F}_{t})_{t\geq0},\mathbb{P} \big); {\mathcal{V}^{\omega}_{t,x}}, W \big]$ be a dissipative measure-valued martingale solution to the system \eqref{P1}.
	Suppose $\mathbf{U}$ be stochastic processes which is adapted to the filtration $(\mathfrak{F}_t)_{t\geq0}$ and which satisfies
	\begin{equation*}
		\begin{aligned}
			\label{operatorBB}
			\mathrm{d}\mathbf{U}  &= \mathbf{U}_1\,\mathrm{d}t  + \mathcal{P}_H\mathbf{U}_2\,\mathrm{d}W,
		\end{aligned}
	\end{equation*}
	with
	\begin{align}\label{eq:smooth}
		\begin{aligned}\vc{U} \in C([0,T]; C_{\dv}^{1}(\T^3)), \ \quad\text{$\mathbb{P}$-a.s.},\qquad
			\mathbb{E}\bigg[ \sup_{t \in [0,T] } \| \vc{U} \|_{L_{\dv}^{2}(\T^3)}^2\bigg]\textless\,\infty,
		\end{aligned}
	\end{align}
	Moreover, $\vc{U}$ satisfy
	\begin{align}\label{new}
		&  \vc{U}_1\in L^2(\Omega;L^2(0,T;L_{\dv}^{2}(\mt))),\quad  \vc{U}_2\in L^2(\Omega;L^2(0,T;L_2(\mathfrak U;L^2(\T^3)))),
	\end{align}
	$$\int_0^T\sum_{k\geq 1}\|\mathcal{P}_H\vc{U}_2(e_k)\|_{L^2(\mathbb{T}^3)}^2\in L^1(\Omega).$$
	Then the following \emph{relative energy inequality} holds:
	\begin{equation}
		\begin{aligned}
			\label{relativeEntropy}
			&\mathcal{E}_{\mathrm{mv}} \left(\textbf{u} \ \Big|  {\bf U} \right)
			(t) \leq
			\mathcal{E}_{\mathrm{mv}} \left(\textbf{u} \ \Big| {\bf U} \right)(0) +M_{RE}(t)  + \int_0^t\mathcal{R}_{\mathrm{mv}} \big(\textbf{u} \left\vert \right.  \mathbf{U}  \big)(s)\,\mathrm{d}s
		\end{aligned}
	\end{equation}
	$\mathbb P$-a.s., where
	\begin{equation}
		\begin{aligned}
			\label{remainderRE}
			\mathcal{R}_{\mathrm{mv}} \big({\bf u} \left\vert \right.\mathbf{U}  \big) 
			&=
			\int_{\mathbb{T}^3} \left\langle {\mathcal{V}^{\omega}_{t,x}}; {{\bf u}\otimes {\bf u} } \right\rangle: \nabla_x \textbf{U} \,\mathrm{d}x+\int_{\mathbb{T}^3}\left\langle {\mathcal{V}^{\omega}_{t,x}};\bf u\right\rangle\cdot {\bf U}_1dxdt -\int_{\T^3} \nabla_x \mathbf{U} : d\lambda_{\mathcal{C}} + \frac12\int_{\T^3} d \lambda_{\mathcal{D}}\\
			&\qquad+\frac{1}{2}
			\sum_{k\in\mathbb{N}}
			\int_{\mathbb{T}^3} \big \langle {\mathcal{V}^{\omega}_{t,x}}; \big\vert {\mathcal{\mathcal{\sigma}}_k({\bf u})} -\mathbf{U}_2(e_k)  \big\vert^2\big\rangle \,\mathrm{d}x 
		\end{aligned}
	\end{equation}
	Here $M_{RE}$ is a real valued square integrable martingale, and the norm of this martingale depends only on the norms of $\mathbf{U}$ in the aforementioned spaces.
\end{Proposition}
\begin{proof}
	
	Here we shall complete proof of mention result in following several steps.
	\medskip 
	
	\noindent {\bf Step 1:} 
	In order to compute $\D\int_{\T^3}{ \big \langle {\mathcal{V}^{\omega}_{t,x}}; \vu  \big \rangle \,\cdot {\bf U} }\D x$, we first recall that $\mathbf{V}=\langle {\mathcal{V}^{\omega}_{t,x}}; \vu  \big \rangle $ satisfies the hypotheses of Lemma \ref{Ito}. Therefore we can apply the Lemma \ref{Ito} to conclude that $\mathbb{P}$-almost surely, for all $\tau\in[0,T]$
	\begin{equation} \label{I1}
		\begin{split}
			 \int_{\T^3}{ \big \langle {\mathcal{V}^{\omega}_{\tau, x}}; \vu  \big \rangle \cdot {\bf U}(\tau) }\dx\dt& =  \int_{\T^3}{ \big \langle {\mathcal{V}^{\omega}_{0,x}}; \vu  \big \rangle \cdot {\bf U}(0) }\dx\dt + \int_0^\tau\int_{\T^3}  \nabla_x {\bf U}: d\mu_C \,dt \\&\qquad +\int_0^\tau\int_{\T^3}{ \left[ \big \langle {\mathcal{V}^{\omega}_{t,x}}; \vu  \big \rangle \cdot {\bf U}_1(t)
				+ \left\langle {\mathcal{V}^{\omega}_{t,x}}; {\vu \otimes \vu }\right\rangle: \nabla_x \vu   \right] }\dx  {\rm d}t \\
			& \qquad +  \sum_{k\geq1}\int_0^\tau\int_{\T^3}{ \mathcal{P}_H{\bf U}_2(e_k) \cdot\mathcal{P}_H \big \langle {\mathcal{V}^{\omega}_{t,x}}; \sigma_k (\vu ) \big\rangle}\dx\, {\rm d}t+  \D \mathcal{M}_1,
		\end{split}
	\end{equation}
	where the square integrable martingale $\mathcal{M}_1(t)$ is given by
	\[
	\mathcal{M}_1(t) = \int_{\T^3} \int_0^t  \vc{U}\cdot\big \langle {\mathcal{V}^{\omega}_{t,x}}; \mathcal{\sigma}_k (\vu ) \big\rangle \, \D W \,dx+ \int_0^t \int_{\T^3}{ \big \langle {\mathcal{V}^{\omega}_{t,x}}; \vu  \big \rangle \cdot \mathcal{P}_H\vc{U}_2 } \,\D W
	\]
	\noindent
	{\bf Step 2:}
	Next, we see that $\p$-a.s., for all $\tau\in[0,T]$
	\begin{equation} \label{I21}
		\begin{split}
			 \int_{\T^3}{ \frac{1}{2} |\vc{U}(\tau)|^2 }\dx&= \int_{\T^3}{ \frac{1}{2} |\vc{U}(0)|^2 }\dx+ \frac{1}{2}\int_0^\tau \sum_{k\geq1}\int_{\T^3}{|\mathcal{P}_H {\bf U}_2(e_k)|^2 }\dx \ {\rm d}t + {\rm d}\mathcal{M}_2,
		\end{split}
	\end{equation}
	where
	\[
	\mathcal{M}_2(t) = \int_0^t \int_{\T^3}{  \vc{U} \cdot\mathcal{P}_H\vc{U}_2 } \, \D W.
	\]
	\noindent 
	{\bf Step 3:}
	We have from energy inequality, $\p$-a.s., for all $\tau\in[0,T]$,
	\medskip
	\begin{align}
		\label{ph}
		\mathfrak{E}(\tau+)\, \leq\, &\mathfrak{E}(0)+ \frac{1}{2} \sum_{k\ge\,1}\int_0^{\tau} \int_{\T^3} \left\langle \mathcal{V}^{\omega}_{s,x};{|{\mathcal{\sigma}_k( u)}|^2} \right\rangle \dx{\rm d}s-\frac{1}{2}\sum_{k\ge\,1}\int_0^\tau\int_{\mathbb{T}^3}\Big(\mathcal{Q}_H\big\langle \mathcal{V}^{\omega}_{s,x};{|{\mathcal{\sigma}_k( u)}|}\big\rangle\Big)^2\dx{\rm d}s \notag \\&\qquad+ \frac12\int_0^{\tau} \int_{\T^3} d \mu_D + \int_0^{\tau}  d\mathcal{M}^2_{E}. 			
	\end{align}
	We now manipulate the product term in the equality \eqref{I1} using properties of projections $\mathcal{P}_H$ and $ \mathcal{Q}_H$. Indeed, note that
	\begin{align*}
		&\int_{\T^3}{ \mathcal{P}_H{\bf U}_2(e_k) \cdot\mathcal{P}_H \big \langle {\mathcal{V}^{\omega}_{t,x}}; \mathcal{\sigma}_k (\vu ) \big\rangle}\\
		&\qquad =\int_{\T^3}{ {\bf U}_2(e_k) \cdot\big \langle {\mathcal{V}^{\omega}_{t,x}}; \sigma_k (\vu ) \big\rangle}-\int_{\T^3}{ \mathcal{Q}_H{\bf U}_2(e_k) \cdot\mathcal{Q}_H \big \langle {\mathcal{V}^{\omega}_{t,x}}; \sigma_k (\vu ) \big\rangle},
	\end{align*}
	and
	\begin{align*}
		\int_{\T^3}{|\mathcal{P}_H {\bf U}_2(e_k)|^2 }=\int_{\T^3}{|{\bf U}_2(e_k)|^2 }-\int_{\T^3}{|\mathcal{Q}_H {\bf U}_2(e_k)|^2 }.
	\end{align*}
	These properties of projections imply that
	\begin{align}\label{rt}
		\begin{aligned}
			&\frac{1}{2} \sum_{k\geq1}\int_{\T^3}{|\mathcal{P}_H {\bf U}_2(e_k)|^2 } -\sum_{k\geq1}\int_{\T^3}{ \mathcal{P}_H{\bf U}_2(e_k) \cdot\mathcal{P}_H \big \langle {\mathcal{V}^{\omega}_{t,x}}; {\sigma}_k (\vu ) \big\rangle}\\&\qquad\qquad+ \frac{1}{2} \sum_{k\ge\,1}\int_s^{t} \int_{\T^3} \left\langle \mathcal{V}^{\omega}_{s,x};{|{\vc \sigma_k( u)}|^2} \right\rangle \dx\,{\rm d}\tau-\frac{1}{2}\sum_{k\ge\,1}\int_s^t\int_{\mathbb{T}^3}\Big(\mathcal{Q}_H\big\langle \mathcal{V}^{\omega}_{s,x};{{\sigma_k( u)}}\big\rangle\Big)^2\dx\,{\rm d}\tau\\&=\frac{1}{2}\sum_{k\ge\,1}\int_{\mathbb{T}^3} \Big \langle {\mathcal{V}^{\omega}_{t,x}}; \big\vert {\mathcal{\sigma}_k({\bf u})} -\mathbf{U}_2(e_k)  \big\vert^2\Big\rangle \,\mathrm{d}x-\frac{1}{2}\sum_{k\ge\,1}\int_{\mathbb{T}^3}\Big|\mathcal{Q}_H\mathbf{U}_2(e_k)-\mathcal{Q}_H \big \langle {\mathcal{V}^{\omega}_{t,x}}; {\sigma}_k (\vu ) \big\rangle\Big|^2\,\dx\\
			&\le\,\frac{1}{2}\sum_{k\ge\,1}\int_{\mathbb{T}^3} \Big \langle {\mathcal{V}^{\omega}_{t,x}}; \big\vert {\mathcal{\sigma}_k({\bf u})} -\mathbf{U}_2(e_k)  \big\vert^2\Big\rangle \,\mathrm{d}x.
		\end{aligned}
	\end{align}
	Finally, in view of the above observations given by (\ref{I1})-(\ref{rt}), we can now add the resulting expressions to establish (\ref{relativeEntropy}). Note that the square integrable martingale $\mathcal{M}_{RE}(t)$ is given by $\mathcal{M}_{RE}(t):= \mathcal{M}_1(t) + \mathcal{M}_2(t)+\mathcal{M}_E^2 (t)$.
\end{proof}
\subsection{Proof of weak-strong principle \ref{Weak-Strong Uniqueness_01}}
	Since $\bar{\vu}$ is the strong pathwise solution to stystem \eqref{P1}, so taking $\mathbf{U}$ = $\bar{\vu}$ in the relative energy inequality \eqref{relativeEntropy}. Then we get  $\p$-a.s., for all $t\in[0,T]$,
	\begin{align}\label{relative energy 3}
		\mathfrak{E}_{\mathrm{mv}}\big(\mathbf{u}\,\big|\bar{\vu}\big)(t\wedge\mathfrak{t}_R)\le\,\mathfrak{E}_{\mathrm{mv}}\big(\mathbf{u}\,\big|\,\bar{\vu}\big)(0)+\mathcal{M}_{RE}(t\wedge\mathfrak{t}_R)+\int_{0}^{t\wedge\mathfrak{t}_R}\mathfrak{R}_{\mathrm{mv}}\big(\mathbf{u}\,\big|\,\bar{\vu}\big)(\tau)\D \tau,
	\end{align}
	where $\mathfrak{R}_{\mathrm{mv}}\big(\mathbf{u}\,\big|\bar{\vu}\big)$ is given by
	\begin{equation}
		\begin{aligned}
			\notag
			\mathfrak{R}_{\mathrm{mv}}\left({\bf u}  \big|\bar{\vu}\right) 
			=&\,  \int_{\mathbb{T}^3} \left\langle {\mathcal{V}^{\omega}_{t,x}}; \left|{(\vu - \bar{\vu})\otimes ( \bar{\vu}-\vu ) }\right| \right\rangle |\nabla_x \bar{\vu}| \,dx +
			\frac{1}{2}
			\sum_{k\in\mathbb{N}}
			\int_{\mathbb{T}^3} \big \langle {\mathcal{V}^{\omega}_{t,x}}; \big\vert {\mathcal{\sigma}_k(\mathbf{u})}  -{\mathcal{\sigma}_k(\bar{\vu})} \big\vert^2 \big \rangle \,\mathrm{d}x \\
			&\qquad+  \int_{\T^3} |\nabla_x \bar{\vu}|\cdot d|\lambda_\mathcal{C}|+  \int_{\T^3} d|\lambda_\mathcal{D}|.
		\end{aligned}
	\end{equation}
Now use the following facts  $$\|\mathbf{\bar{\vu}}\|_{W^{1,\infty}(\mathbb{T}^3)}\le\,c(R)\,\, \text{for all}\,\, t\,\le\,\tau_R \,\,\&\,\,\,\left| {(\vu - \bar{\vu}) \otimes ( \bar{\vu}-\vu )} \right| \leq |\vu -\bar{\vu}|^2,$$
\begin{equation}
	\begin{aligned}
		\notag
		\sum_{k \ge 1}\big\vert {\mathcal{\sigma}_k(\mathbf{u})}  -{\mathcal{\sigma}_k(\bar{\vu})} \big\vert^2\leq D_1|\mathbf{u}-\bar{\vu}|^2,
	\end{aligned}
\end{equation}
to conclude that
\begin{align*}
	\frac{1}{2}
	\sum_{k\in\mathbb{N}}
	\int_{\mathbb{T}^3} \big \langle {\mathcal{V}^{\omega}_{t,x}}; \big\vert {\mathcal{\sigma}_k(\mathbf{u})}  -{\mathcal{\sigma}_k(\bar{\vu})} \big\vert^2 \big \rangle \,\mathrm{d}x
	\leq c(L)\,
	\mathfrak{E}_{\mathrm{mv}} \big( {\mathbf{u}}\left\vert \right. \bar{\vu}  \big),
\end{align*}
and
	\begin{align}\label{reminder 2}
		\int_{0}^{t\wedge\mathfrak{t}_R}\mathfrak{R}_{\mathrm{mv}}\big(\mathbf{u}\,\big|\bar{\vu}\big)(s)\,\D s\le\,c(R)\int_{0}^{t\wedge\mathfrak{t}_R}\big(\mathfrak{E}_{\mathrm{mv}}\big(\mathbf{u}\,\big|\,\bar{\vu}\big)(s)\, \D s.
	\end{align}
	In view of \eqref{relative energy 3} and \eqref{reminder 2}, a simple consequence of Gronwall's lemma gives, for all $t\in[0,T]$
	\begin{align*}
		\mathbb{E}\big[\mathfrak{E}_{\mathrm{mv}}\big(\mathbf{u}\,\big|\bar{\vu}\big)(t\wedge\mathfrak{t}_R)\big]\le c(R)\,\mathbb{E}\big[\mathfrak{E}_{\mathrm{mv}}\big(\mathbf{u}\,\big|\bar{\vu}\big)(0)\big].
	\end{align*}
	Since, initial data are same for both solutions, right hand side of above inequality equals to zero. Therefore it implies that for all $t\in [0,T]$
	$$\mathbb{E}\big[\mathfrak{E}_{\mathrm{mv}}\big(\mathbf{u}\,\big|\bar{\vu}\big)(t\wedge\mathfrak{t}_R)\big]=0.$$
	This also implies that
	$$\int_0^T\mathbb{E}\big[\mathfrak{E}_{\mathrm{mv}}\big(\mathbf{u}\,\big|\bar{\vu}\big)(s\wedge\mathfrak{t}_R)\big]\D s=0.$$
	In view of a priori estimates, a usual Lebesgue point argument, and application of Fubini's theorem reveals that $\p$-a.s.,
	$$\int_0^T\mathfrak{E}_{\mathrm{mv}}\big(\mathbf{u}\,\big|\,\bar{\mathbf{u}}\big)(t\wedge\mathfrak{t}_R)\dt=0.$$
	Since, the defect measure $\mathcal{H}\ge\,0$, we have $\p$-a.s., for a.e. $t\in[0,T]$, $\mathcal{H}(t\wedge\mathfrak{t}_R)=0$,
	 Moreover, $\p-\mbox{a.s.}$
	\begin{equation*}
		\mathcal{V}^{\omega}_{t \wedge \mathfrak{t}_R,x} = \delta_{\bar{\bf u}(t \wedge \mathfrak{t}_R,x)}, \,\,  \mbox{for a.e. }(t,x)\in (0,T)\times \T^3.
	\end{equation*} 
	This proves our claim.

\section{Proof of Theorem~\ref{T_ccE}: Convergence to regular solution}\label{Section 8}
\label{proof2}

We have proven that the approximate solutions $\widetilde{\vu}_n$  to \eqref{approximate equations} for the stochastic incompressible Euler system converges to a {\em dissipative measure-valued martingale} solution, in the sense of Definition~\ref{def:dissMartin}. Employing the corresponding weak (measure-valued)--strong uniqueness results (cf. Theorem~\ref{Weak-Strong Uniqueness_01}), we can show the strong convergence of approximate solutions to the strong solution of the system on its lifespan.

First note that, Proposition~\ref{prop:skorokhod1} and Theorem~\ref{Weak-Strong Uniqueness_01} gives the required weak-$*$ convergence. Indeed, from Proposition~\ref{prop:skorokhod1}, we have $\widetilde{\p}$-a.s.,
$$\widetilde{\vu}_n(\cdot\wedge\mathfrak{t}_R)\to\langle {\widetilde{\mathcal{V}}^{\omega}_{t,x}}; \widetilde{\vu} \rangle(\cdot\wedge\mathfrak{t}_R)\,\,\mbox{in}\,\,\,C_w([0,T],L_{\dv}^2(\mathbb{T}^3)),$$
Combination of above convergence and Theorem~\ref{Weak-Strong Uniqueness_01} gives the required weak-* convergence. For the proof of strong convergence of momentum in $L^1(\T^3)$, note that from Proposition \ref{prop:skorokhod1}, Theorem \ref{Weak-Strong Uniqueness_01}, energy bounds \eqref{uniform estimate} and using the fact limit Young measure of any subsequence $\delta_{\widetilde{\vu}_{n_k}(\cdot\wedge\mathfrak{t}_R)}$ is $\delta_{\bar\vu(\cdot\wedge\mathfrak{t}_R)}$. We have $\widetilde{\mathbb{P}}$-a.s., sequence of young measure converges to dirac Young measure, i.e. $\widetilde{\p}$-a.s.
$$ \delta_{\widetilde{\vu}_{n}(\cdot\wedge\mathfrak{t}_R)}  \rightarrow \delta_{\bar{\vu}(\cdot \wedge \mathfrak{t}_R)},\,\, \text{weak-$*$ in}\, \, L^{\infty}((0,T)\times \T^3; \mathcal{P}(\R^4))$$
By theory of Young measure \cite[Proposition 4.16]{Balder}, it implies that, $\widetilde{\p}$-a.s. $\widetilde{\vu}_{n}(\cdot\wedge\mathfrak{t}_R)$ converges to $\bar{\vu}(\cdot \wedge \mathfrak{t}_R)$ in measure respectively. Note that, $\widetilde{\p}$-a.s. sequence $\widetilde{\vu}_{n}(\cdot\wedge\mathfrak{t}_R)$ is uniformly integrable and converges in measure, therefore Vitali's convergence theorem implies that $\widetilde{\p}$-a.s,
\begin{equation*}
	\begin{aligned}
		\widetilde{\vu}_{n}(\cdot\wedge\mathfrak{t}_R) &\to \bar\vu(\cdot\wedge\mathfrak{t}_R) \ \mbox{ strongly in}\ L^1((0,T) \times \T^3),
	\end{aligned}
\end{equation*}
This finishes the proof of the Theorem \ref{T_ccE}.

\subsection*{Acknowledgments}
The author wishes to thank Ujjwal Koley for many stimulating discussions and valuable suggestions. 

\end{document}